%% file: PurityMain.tex
\documentclass{amsart}

\usepackage{hyperref}
\usepackage{amssymb}
\usepackage{amsmath}
\usepackage[capitalize]{cleveref}
\usepackage[style=alphabetic,natbib=true,url=false,maxnames=10]{biblatex}

\usepackage{mathtools}
\usepackage[cmtip,all]{xy}
\usepackage{tikz-cd}
\usepackage{harmony}
\usepackage{scalerel}

\input{symbols}
\bibliography{litteratur}

\numberwithin{equation}{section}

\begin{document}

\title[Purity in the tame cohomology]{Purity in the tame cohomology}


\author{Amine Koubaa}
\email{koubaa@math.uni-hannover.de}
\address{Institut für Algebraische Geometrie, Leibniz Universität Hannover, Welfengarten 1, 30167 Hannover Germany}
\keywords{Tame cohomology, Cartier operator, logarithmic differentials, logarithmic schemes, adic spaces}
\thanks{ Funded by the Deutsche Forschungsgemeinschaft (DFG, German Research Foundation) TRR 326 \textit{Geometry and Arithmetic of Uniformized Structures}, project number 444845124.}


\date{\today}


\begin{abstract}
Let $(X,Z)$ be a regular pair of pure codimension $r$ over a base scheme $S$ of characteristic $p$.  Under the assumption of resolution of singularities, we prove purity for the tame cohomology with coefficient in the logarithmic de Rham-Witt sheaves $\nu_m(n)$ for the regular pair $(X,Z)$, i.e. the existence of a natural isomorphism
\[
Ri^!\nu_{X,m}(n)\cong \nu_{Z,m}(n-r)[-r]
\]
for all $m,n\geq 0$, where the logarithmic de Rham-Witt sheaves $\nu_{X,m}(n)$ are the Frobenius-fixed elements in the de Rham-Witt sheaves $W_m\Omega^{n}_{X/\mF_p}$. 
\end{abstract}

\maketitle
\tableofcontents
\section{Introduction}
\subsection{\texorpdfstring{Purity in the prime to $p$ case}{Purity in the prime to p case}}
\input{GysinPrimeTop}

\subsection{Semi-purity of the étale cohomology with coefficients in the log de Rham-Witt sheaves}
\input{DualizingSheafNu}
\subsection{Purity of the tame cohomology in positive characteristic}
\input{MainTheoremIntro}
\section{Logarithmic schemes and their differentials}
\input{LogSchemes}
\section{Cohomology of discretely ringed adic spaces}\label{ChapterAdic}
\input{TameRZLet}
\section{Differentials over adic spaces}
\input{OmegaPlus}
\section{Gysin morphism and purity in positive characteristic}
\input{PurityPrime}

\appendix
\section{Compactifications}\label{sec:Compactifications}
\input{Resolution}

\printbibliography
\end{document}

%% file: symbols.tex
\def\Spa{\textrm{Spa}}

\def\Spec{\textrm{Spec}}
\def\Op{\textrm{Op}}

\def\Sh{\textrm{Sh}}
\def\gp{\textrm{gp}}
\def\et{\textrm{\'et}}

\def\RZ{\textrm{RZ}}

\theoremstyle{definition}
\newtheorem{theorem}{Theorem}[section]
\newtheorem{definition}[theorem]{Definition}
\newtheorem{example}[theorem]{Example}

\newtheorem{proposition}[theorem]{Proposition}
\newtheorem{lemma}[theorem]{Lemma}
\newtheorem{remark}[theorem]{Remark}
\newtheorem{corollary}[theorem]{Corollary}
\newtheorem{conjecture}[theorem]{Conjecture}
\newtheorem{DefLem}[theorem]{Definition/Lemma}
\newtheorem{notation}[theorem]{Notation}
\newtheorem{claim}[theorem]{Claim}

\newcommand\llog{\mathrm{log}}
\newcommand\dlog{\mathrm{dlog}}
\newcommand\supp{\mathrm{supp}}
\newcommand\tame{\textrm{t}}
\newcommand\sett{\textrm{s\'et}}
\newcommand{\mX}{\mathcal{X}}
\newcommand{\mY}{\mathcal{Y}}
\newcommand{\mO}{\mathcal{O}}
\newcommand{\mZ}{\mathcal{Z}}
\newcommand{\mU}{\mathcal{U}}

\newcommand{\mm}{\mathfrak{m}}

\newcommand{\mM}{\mathcal{M}}

\newcommand{\mF}{\mathbb{F}}
\newcommand{\mN}{\mathbb{N}}

\newcommand{\mbZ}{\mathbb{Z}}

\DeclareMathOperator*{\colim}{\mathrm{colim}}
\newcommand{\nn}{\textbf{n}}

\newcommand{\lcirc}[1]{\scaleobj{.8}{\mbox{\Kr{#1}}}}

%% file: GysinPrimeTop.tex
Let $X$ and $Z$ be regular schemes, and let $i: Z \hookrightarrow X$ be a closed immersion of pure codimension $r$. We refer to the pair $(X,Z)$ as a \emph{regular pair} of codimension $r$. If $X$ and $Z$ are smooth over a base scheme $S = \Spec(k)$, where $k$ is a field and $i:Z\hookrightarrow X$ is a closed immersion, then we call $(X,Z)$ a \emph{smooth pair} over $k$.
Given a regular pair $(Z, X)$ of codimension $r$, the theorem of absolute cohomological purity, due to Gabber and Fujiwara \cite{FujiwaraPurity}, asserts the following:

\begin{theorem}[\cite{FujiwaraPurity}]\label{PurityPrimeToP}
	Let $(X,Z)$ be a regular pair of codimension $r$, and let $i: Z \hookrightarrow X$ be the associated closed immersion. Let $l$ be an integer which is invertible on $X$. Then there exists a canonical isomorphism in the derived category of étale sheaves:
	\begin{equation}\label{PurityPrimetoPIso}
		Ri^! \mathbb{Z}/l(n)[2r] \cong i^* \mathbb{Z}/l(n - r).
	\end{equation}

\end{theorem}

%% file: DualizingSheafNu.tex
Assume that $(X,Z)$ is a smooth pair over a perfect field $k$ of characteristic $p$. Let $S:=\Spec(k)$. Our goal is to construct a purity isomorphism analogous to \eqref{PurityPrimetoPIso}, replacing $\mathbb{Z}/l$ with $\mathbb{Z}/p^m$. However, defining $\mathbb{Z}/p^m(n)$ as $\mu_{p^m}^{\otimes n}$ does not make sense, since the group scheme $\mu_{p^m} \cong \Spec(k[T]/(T^{p^m} - 1))$ is not an \'etale sheaf.


In \cite[Theorem 2.1]{MilneValuesOfZeta} and \cite[Théorème 3.5.8]{CohomologyLogGros}, Milne and Gros presented an alternative, namely the logarithmic de Rham-Witt sheaf $\nu_m(n)$ (cf. \cref{Def:LogDeRhamWitt})  and proved a semi-purity theorem for it. Namely that canonical isomorphisms
\begin{equation}\label{eq:semi-purityEtale}
	R^ji^!\nu_{X,m}(n)=\begin{cases}
		0 &\textrm{ if } j\neq r,r+1\\
		\nu_{Z,m}(n-r) &\textrm{ if } j=r
	\end{cases}
\end{equation}
exist. 

The remaining term $R^{r+1}i^!\nu_{X,m}(n)$ is the obstruction to full purity, and it may fail to vanish, as we show in \cref{Ex:CounterExampleVanishing}.

The logarithmic de Rham-Witt sheaves $\nu_{X,1}(n)$ are characterized by the short exact sequence (cf. \cite{MilneDuality}):
\begin{equation}\label{SESCartierIntro}
	0 \to \nu_{X,1}(n) \to Z \Omega^n_{X/S} \xrightarrow{C - 1} \Omega^n_{X/S} \to 0,
\end{equation}
where $C$ denotes the Cartier operator (see \cite[Lemma 1.1]{MilneDuality} for its definition and existence) and $Z \Omega^n_{X/S}$ is the sheaf of closed $n$-forms.
\begin{example}\label{Ex:CounterExampleVanishing}
	Consider the case where $X = \Spec(k[T])$ and $Z = V(T)$. Then $R^2 i^! \nu_{X,1}(0)=R^2 i^! \mbZ/p\mbZ$ is non-zero.
\end{example}

\begin{proof}[Proof that $R^2 i^! \nu_{X,1}(0)$ is non-zero]
		The claim is étale local, hence we may assume that $X$ is strictly henselian. Let $A:= k[T]^{\mathrm{sh}}_{(T)}$ be a strict henselization of $k[T]$ at the ideal $(T)$ and replace $X$ with $\Spec(A)$.
		As in the proof of \cite[Theorem 2.1]{MilneValuesOfZeta}, the long exact sequence of cohomology groups associated to \eqref{SESCartierIntro} is
		\[
		0\to R^{1}i^!\nu_{X,1}(n)\to R^1i^!Z\Omega_{X/k}^n\xrightarrow{C-1} R^1i^!\Omega_{X/k}^n\to R^{2}i^!\nu_{X,1}(n)\to 0.
		\]
		Hence it is enough to prove that $R^1i^!Z\Omega_{X/k}^n\xrightarrow{C-1} R^1i^!\Omega_{X/k}^n$ is not surjective.
		
	When $n=0$, $\Omega^0_{X/k} = \mathcal{O}_X$ and $Z\Omega^0_{X/k} = \mathcal{O}_X^p$. Thus $R^1 i^! Z \Omega^0_{X/k}$ and $R^1 i^! \Omega^0_{X/k}$ are isomorphic to the quasi-coherent sheaves associated to $(A[T^{-1}])^p/A^p$  and $A[T^{-1}]/A$ respectively.

	The map $C - 1$ is identified with the Artin-Schreier map $$F^{-1} - 1 : (A[T^{-1}])^p/A^p \to A[T^{-1}]/A, \quad \gamma^p\mapsto \gamma-\gamma^p.$$ Since the equation $\gamma-\gamma^p = 1/T$ has no solution $\gamma$ in $A[T^{-1}]^p$, the element $1/T$ does not lie in the image of $F^{-1} - 1$, and hence $R^2 i^! \nu_{X,1}(0)$ is non-zero.
\end{proof}

%% file: MainTheoremIntro.tex
As purity fails in the \'etale topology, we  work  with the \emph{tame topology} instead, defined in \cite{HSTame}. Roughly speaking, this is the topology whose covers are \'etale morphisms that are tamely ramified at the boundary. For instance, Artin--Schreier covers that ramify at the boundary are excluded. For more details, see \cite{HSTame} or \cref{ChapterAdic} of this document. For a scheme $X$ over a field $k$, we denote the tame site by $(X/k)_t$.

We now state the main theorem of this article as well as a sketch of the proof.

\begin{theorem}\label{MainTheorem}
Let $X$ be a smooth scheme over a perfect field $k$. Let $i: Z \hookrightarrow X$ be a smooth closed immersion of pure codimension $r$. Assume that resolution of singularities and embedded resolution of singularities (as stated in \cref{ConjResolution3} and \cref{GeneralResolution}) hold. Then there is an isomorphism in the derived category of tame sheaves $D((Z/k)_t)$:
\[
\nu_{Z, m}(n - r)[-r] \xrightarrow{\sim} Ri^! \nu_{X, m}(n).
\]
\end{theorem}
\begin{remark}
	The following sketch of the proof uses discretely ringed adic spaces, introduced in \cref{ChapterAdic}. We define $\mathcal{X} := \Spa(X, k)$ and $\mathcal{Z} := \Spa(Z, k)$. The idea of the proof should remain accessible even without a full familiarity with adic spaces.

	The adic version of the theorem states that
	\[
	\nu_{\mathcal{Z}, m}(n - r)[-r] \xrightarrow{\sim} Ri^! \nu_{\mathcal{X}, m}(n)
	\]
	is an isomorphism in $D(\mathcal{Z}_t)$. The two formulations are equivalent by \cite[§14]{HSTame}.
\end{remark}

\begin{proof}[Sketch of the proof]
	\begin{enumerate}
			\item Under the assumption of resolution of singularities, we prove the existence of the short exact sequence (\cref{C-1Shortexact})
		\[
		0 \to \nu_{\mathcal{X}}(n) \to Z\Omega^{+,n}_{\mathcal{X}/k} \xrightarrow{C - 1} \Omega^{+,n}_{\mathcal{X}/k} \to 0.
		\]
		The key input is that $\Omega^{+,n}(\Spa(X, \bar{X})) = \Omega^{n,\log}(X, \bar{X})$ by the main theorem of \cite{LogDiff} whenever $(X, \bar{X})$ is a log smooth log scheme. It is well-known that $\Omega^{n,\log}(X, \bar{X})$ fits into such a short exact sequence (see \cref{C-1ShortExactLog}).

			Note that the analogous sequence for $\Omega^n$ is not exact: the preimage of $f \in \mathcal{O}_X(X)$ under $C - 1$ lies in $\mathcal{O}_Y(Y)$, where $Y = \Spec_{\mathcal{O}_X} \mathcal{O}_X[T]/(T^p - T - f)$ is an Artin--Schreier cover of $X$. Such covers are not tamely ramified when $f$ is not invertible on $\bar{X}$.
					
			\item Prove that $Ri^!\Omega^{+,n}_{\mathcal{X}/k}\cong \Omega^{+,n-r}_{\mathcal{Z}/k}[-r]$ (\cref{ValuesOfRi!Omega+}) and \break $Ri^!Z\Omega^{+,n}_{\mathcal{X}/k}\cong Z\Omega^{+,n-r}_{\mathcal{Z}/k}[-r]$ (\cref{ValueOfRi!ZOmega}).
		
		 To achieve this, we do the following: we check that the tame cohomology of $\Omega^{+,n}$ and $Z\Omega^{+,n}$ is equal to the cohomology of the underlying topological space (\cref{StronglyEtaleEqualOpen}). For $\mathcal{F}=\Omega^{+,n}$ and $\mathcal{F}=Z\Omega^{+,n}$ we prove that  
		\[H^i(\Spa(X,S),\mathcal{F})\cong \textrm{colim}_{\bar{Y}} H^i(\bar{Y},\mathcal{F}|_{\bar{Y}}),
		\]
		where $\bar{Y}$ are $S$-compactifications of $X$.
		
		Assuming resolution of singularities, we may compute this using a system of $\bar{Y}$'s for which $\Omega^{+,n}_{\Spa(Y,\bar{Y})}|_{\bar{Y}}=\Omega^{n,\log}_{(Y,\bar{Y})}$ holds. We finally use the vanishing of the higher cohomology groups of $\Omega^{n,\log}_{(Y,\bar{Y})}$ when $X$ and $S$ are affine (\cref{AcyclicityProposition}) to deduce the vanishing of $R^si^!\Omega^{+,n}$ and $R^si^!Z\Omega^{+,n}$ for $s\neq \textrm{codim}_X(Z)$.
		
		\item Apply the long exact sequence of $Ri^!$ to obtain the diagram \refeq{FinalDiagram}
		\[
		\begin{tikzcd}
			0\ar[r] &\nu_{\mathcal{Z}}(n-r)\ar[r]\arrow[dotted,d]&Z\Omega_{\mathcal{Z}}^{+,n-r}\ar[r,"C-1"]\arrow[d,"\cong"] &\Omega^{+,n-r}_{\mathcal{Z}}\ar[d,"\cong"]\ar[r]&0\\
			0\ar[r] &R^ri^!\nu_{\mathcal{X}}(n)\ar[r]&R^ri^!Z\Omega_{\mathcal{X}}^{+,n}\ar[r,"C-1"]&R^ri^!\Omega_{\mathcal{X}}^{+,n}\ar[r]&0,
		\end{tikzcd}
		\] and conclude that 
		\[
		\nu_{\mathcal{Z}}(n-r)[-r]\to Ri^!\nu_{\mathcal{X}}(n)
		\]
		is an isomorphism.
		\item Use the short exact sequence in \cref{highermSES}
		\[
		0\to \nu_{\mathcal{Z},m}(n)\to \nu_{\mathcal{Z},m+1}(n)\to \nu_{\mathcal{Z},1}(n)\to 0,
		\] 
		to deduce by induction the isomorphism
		\[
		\nu_{\mathcal{Z},m}(n-r)[-r]\to Ri^!\nu_{\mathcal{X},m}(n).\qedhere
		\]
	\end{enumerate}
\end{proof}
Finally, following the work of Shiho in \cite{Shiho}, we generalize \cref{MainTheorem} to regular pairs over basis of characteristic $p$.
\begin{corollary}(cf. \cref{regularCase})
		Let $i: Z\to X$ be a closed immersion of  schemes over a  scheme $S$ of characteristic $p$. Assume that $(X,Z)$ is a regular pair of pure codimension $c$. Then  in $D((Z/S)_t)$ we have the isomorphism 
	\[
	Ri^!\nu_{m,X}(n)\cong \nu_{m,Z}(n-r)[-r]
	\]
	under the assumption of resolution of singularities.
\end{corollary}
\begin{proof}[Sketch of the proof]
	Apply Popescu theorem (\cref{PopescuTheorem}) in order to reduce to the smooth case over $\mF_p$, which is solved by \cref{MainTheorem}.
\end{proof}

\paragraph{Acknowledgment}
I thank Katharina Hübner for her ideas and the endless discussions about every detail of this article. I also thank Piotr Achinger and his group, whom I visited in Warsaw, and who helped me clarify some of the log geometric aspects of this work. I also thank Alberto Merici and Morten Lüders for sharing and explaining their respective work on the matter.

%% file: LogSchemes.tex
Logarithmic schemes provide a natural framework for defining logarithmic differentials, which arise in the study of purity for $p$-torsion sheaves in characteristic $p$. This section recalls the basic definitions and properties of logarithmic schemes and their differentials, and discusses normal crossing divisors and the acyclicity of blow-ups, which are key tools for our proof of purity.

For more details on logarithmic schemes and their differentials, we refer the reader to \cite{Ogus}.
\subsection{Logarithmic schemes}
\input{LogarithmicSchemes}

\subsection{Logarithmic differentials}
\input{LogDiff}
\subsection{Normal Crossing Divisors}
\input{NormalCrossing}

%% file: LogarithmicSchemes.tex
In this paper, a monoid is assumed to be commutative and to have a unit element. 
\begin{definition}\label{def:LogScheme}
	Let $X$ be a scheme. A \textit{prelog structure} on $X$ is a morphism of sheaves of monoids
	\[
	\alpha_X:\mathcal{M}_X\to \mathcal{O}_X
	\]
	on the small \'etale site $X_\et$, where $\mathcal{O}_X$ is considered as a multiplicative monoid.  We say that $\alpha_X$ is a \emph{log structure} if 
	\[
	\alpha_X^{-1}(\mathcal{O}_X^{\times})\xrightarrow{\alpha_X} \mathcal{O}_X^{\times}
	\]
	is an isomorphism. A log scheme is a pair $(X,\mathcal{M}_X)$ where $X$ is a scheme and $\alpha_X:\mathcal{M}_X\to \mathcal{O}_X$ is a log structure. Here the morphism $\alpha_X$ is a part of the structure but we generally choose to omit it from the notation.
	
	A morphism $(X,\mathcal{M}_X)\to (Y,\mathcal{M}_Y)$ of log schemes is a pair $(f,f^{\flat})$ where $f:X\to Y$ is a morphism of schemes and $f^{\flat}:f^{-1}\mathcal{M}_Y\to \mathcal{M}_X$ is a morphism of monoids which respects the log structure. 
	
	To each prelog structure $\alpha_X:\mathcal{M}_X\to \mathcal{O}_X$ we define the associated log structure as the pushout  $\mathcal{M}^\llog_X:=\mathcal{M}_X\bigoplus_{\alpha_X^{-1}(\mathcal{O}_X^{\times})}\mathcal{O}_X^{\times}$ together with the natural morphism $\tilde{\alpha}_X:\mathcal{M}^\llog_X\to \mathcal{O}_X$.   The log structure $\tilde{\alpha}_X$ is called the \emph{logification} of the prelog structure $\alpha_X$.

	A morphism of logarithmic schemes $Y\to X$ is \emph{strict}, if the log structure of $Y$ is the pullback of the log structure of $X$, i.e. $\mM_Y= f^*\mM_X$ which is the logification of the prelog structure $f^{-1}\mM_X\to f^{-1}\mM_X\to \mM_Y$.
\end{definition}

\begin{example}
	Let $X$ be a scheme and $U\subseteq X$ an open subscheme. Then we define a log structure, called the \textit{compactifying log structure}
	\[
	\alpha_{U/X}:\mathcal{M}_{U/X}\to \mathcal{O}_X
	\]
	as follows. For every étale morphism $V\to X$ we set
	\[
	\mathcal{M}_{U/X}(V)=\{f\in\mathcal{O}_X(V)\mid f|_{U\times_X V}\in\mathcal{O}_X(U\times_X V)^\times\},
	\]
	and $\alpha_{U/X}$ is the natural inclusion $\mathcal{M}_{U/X}(V)\hookrightarrow\mathcal{O}_X(V)$.
\end{example}



\begin{notation}
	Let $X$ be a log scheme. Denote the underlying scheme by $X^\circ$. Define
	\[
	X_{\textrm{tr}}:= \{x\in X^\circ| \mM_{X,x}= \mO_{X,x}^\times\}.
	\]
\end{notation}

%% file: LogDiff.tex
Let $f:X\to S$ be a morphism of prelog schemes. Denote by $\alpha:\mathcal{M}_X\to \mathcal{O}_X$ the structure morphism of $X$. Let $\mathcal{E}$ be an $\mathcal{O}_X$-module. An \textit{$S$-derivation} from $(\mathcal{O}_X,\mM_X)$ to $\mathcal{E}$ is a pair $(d,\dlog)$ of maps, where $d:\mathcal{O}_X\to \mathcal{E}$ is an $S$-derivation and $\dlog:\mathcal{M}_X\to \mathcal{E}$ is a morphism of monoids satisfying:
\begin{itemize}
	\item $d(\alpha(m))=\alpha(m)\dlog(m)$ for each section $m\in \mathcal{M}_X$,
	\item $\dlog(f^\flat(m))=0$ for each section $m\in \mathcal{M}_S$.
\end{itemize}

Note that if $f:X\to S$ is a morphism of log schemes induced by a morphism of prelog schemes $X^{\mathrm{pre}}\to S^{\mathrm{pre}}$, then the derivations are the same. 

\begin{lemma}\label{DefinitionDifferentials}\cite[Chapter IV, Theorem 1.2.4]{Ogus}
	Let $f:X\to S$ be a morphism of prelog schemes. The functor associating to an $\mathcal{O}_X$-module $\mathcal{E}$ the set of $S$-derivations is represented by the $\mathcal{O}_X$-module of \textit{logarithmic Kähler differentials} $\Omega^{\log}_{X/S}$, together with the \textit{universal $S$-derivation} $(d,\dlog)$.
\end{lemma}

\begin{proof}[Proof Idea.]
	The module $\Omega^{\log}_{X/S}$ can be described as the quotient of  
\[
\Omega_{X^{\circ}/S^{\circ}} \bigoplus (\mO_X\otimes_\mbZ \mM_X^\gp)
\]
by the submodule generated by all elements of the form:
 	\begin{itemize}
 		\item $d(\alpha_X(a))-\alpha_X(a)\otimes a$ where $a$ is a section of  $\mathcal{M}_X$
 		\item $0+1\otimes f^\flat (b)$ where $b$ is a section  of $\mathcal{M}_S$.\qedhere
 	\end{itemize}  
\end{proof}
 	Elements of the form $0+1\otimes a$ are denoted by $d\log a$ for a section $a$ of $\mathcal{M}_X$. The first equivalence relation should tautologically correspond to 
\[
\dlog(x)=dx/x.
\]
\begin{notation}
	\begin{itemize}
		\item 	For simplicity, we write $\Omega_{X/S}$ instead of $\Omega_{X^\circ/S^\circ}$.
		\item 	Let $X$ be a scheme over a base scheme $S$. Let $D\subset X$ be a normal crossing divisor. Endow $X$ with the compactifying log structure for $X\backslash D$. Then we sometimes write $\Omega_{X/S}(\llog D)$ instead of $\Omega^{\llog}_{X/S}$.
	\end{itemize}
\end{notation}
\begin{example}[{\cite[Chapter IV, Example 1.2.17.]{Ogus}}]
	Let $f:X\to Y$ be a strict morphism. Then $\Omega^{\log}_{X/Y}\cong \Omega_{X/Y}$. 
\end{example}

%% file: NormalCrossing.tex
Let $D=D_1\cup \cdots \cup D_r$ be a strict normal crossing divisor (snc divisor) on a scheme $X$ \cite[Chapter III, Definition 1.8.1]{Ogus}. We consider the compactifying log structure on $X$ which is associated to the open subscheme $X\backslash D$. By definition, the irreducible components $D_i$ are locally cut out by certain $f_i\in \mathcal{O}_X$, this log structure is locally defined as the logification of the prelog structure:
\[
\bigoplus_{i=1}^r \mathbb{N}\to \mathcal{O}_X,\; (n_1,\dots ,n_r)\mapsto (f_1^{n_1} \cdots f_r^{n_r}).
\]
In this subsection we are interested in properties of such log structures and their differentials.
\begin{lemma}[{\cite[Chapter IV, Corollary 1.2.8]{Ogus}}]
	Let $S$ be a scheme, and let $X$ be a smooth scheme over $S$ endowed with the log structure defined by an snc divisor $D=\bigoplus_{i=0}^nD_i$. Then the sheaf $\Omega_{X}(\textrm{log}(D))$ is a coherent $\mathcal{O}_X$-module.
\end{lemma}

The following proposition describes the structure of logarithmic differentials $\Omega_X(\textrm{log}(D))$ when $D$ is an snc divisor. This is a well-known result, but we could not find a reference for it in the literature, except for the case where $S=\Spec(k)$ and $k=k^{\textrm{alg}}$ is algebraically closed. 
\begin{proposition}[{Residue sequences, cf. \cite[Properties 2.3]{Esnault}}]\label{ResidueSequences}
	Let $X$ be a smooth scheme over $S$. We endow it with the log structure defined by an snc divisor $D=D_1\cup \dots \cup D_n$. Let $a>0$ be an integer. Define $D^1:=D-D_1$. We have the following short exact sequences 
	\begin{equation*}
	\begin{aligned}
		0 &\to \Omega_{X/S} \to \Omega_{X/S}(\textrm{log}(D)) \to \bigoplus_{i=1}^n\mathcal{O}_{D_i} \to 0,\\
		0 &\to \Omega^a_{X/S}(\textrm{log}(D - D_1)) \to \Omega^a_{X/S}(\textrm{log}(D)) \to \Omega_{D_1/S}^{a-1}(\textrm{log}(D^1)|_{D_1}) \to 0,\\
	\end{aligned}
	\end{equation*}
\end{proposition}
\begin{remark}
	We wrote $\mathcal{O}_{D_i}$ and $\Omega_{D_i/S}^{a}$ instead of $i_*\mathcal{O}_{D_i}$ and $i_*\Omega_{D_i/S}^{a}$, where $i:D_i\to X$ is the given closed immersion. This is done to simplify the notation.
\end{remark}
\begin{proof}
	We follow the proof in the reference. The condition that $S$ is the spectrum of an algebraically closed field, as in loc. cit., may be dropped.	

	For the first sequence define the map locally as 
	\begin{align*}
			&\alpha:\Omega_{X/S}(\textrm{log}(D))\to \bigoplus_{i=1}^n\mathcal{O}_{D_i}	\\
			 &a\alpha(dg)+\alpha(\sum_{i=1}^{n}a_i\dlog  (f_i))=\sum_{i=1}^na_i|_{D_i},
\end{align*}	
	where the $f_i$'s locally define the divisors $D_i$. Note that this is well defined and glues to a global map: given other equations $f_i'$ defining $D_i$, we have $$\alpha(a_i\dlog   f_i)=\alpha(a_i \dlog  f_i'+a_i\dlog (f_i/f_i'))=a_i\in \mathcal{O}_{D_i}.$$
	The morphism $\alpha$ is obviously surjective and $\textrm{ker}(\alpha)$ is $\Omega_{X/S}$.
	
	Let $\eta$ be a section of $\Omega_{X/S}^a(\llog D)$. Then $\eta$ is locally a sum of elements of the form 
	\[
	\eta_1 + \dlog   f_1\wedge \eta_2,
	\]
	where $\eta_1$ is a section of  $\Omega_{X/S}^a(\llog (D- D_1))$ and $\eta_2$ is a section of $\Omega_{X/S}^{a-1}(\log (D-D_1))$. We define the map
	\[
	\phi:\Omega_{X/S}^a(\llog D)\to \Omega^{a-1}_{D_1/S}(\llog(D-D_1)|_{D_1}), \quad \eta_1 + \dlog   f_1 \wedge \eta_2 \mapsto \eta_2|_{D_1}.
	\]
	 The map $\phi$ is well defined and independent from the representative $f_1$. Indeed, let $f_1'$ be another local generator of $\mathcal{O}(-D_1)$. Then 
	 \[
	 \eta_1+\dlog  (f_1)\wedge \eta_2= \big(\eta_1+\dlog  (\frac{f_1'}{f_1})\wedge \eta_2 \big)+ \log (f_1')\wedge \eta_2.
	 \]
	  The map $\phi$ is surjective and the elements of the kernel are locally of the form $\eta=\eta_1$. Hence $\ker(\phi)=\Omega_{X/S}^a(\llog( D-D_1))$.
	
\end{proof}
\begin{corollary}\label{ResidueSequenceClosedForms}
	For the same setup as in \cref{ResidueSequences} we have the short exact sequence of closed forms
	\[
	0\to Z\Omega_{X/S}^a(\llog(D-D_1))\to Z\Omega_{X/S}^a(\llog(D))\to Z\Omega^{a-1}_{D_1/S}(\llog(D-D_1)|_{D_1})\to 0.
	\]
\end{corollary}
\begin{proof}
	The morphisms in the second exact sequence in \cref{ResidueSequences} commutes with $d$. We have therefore the morphism of short exact sequences
	\[
	\begin{tikzcd}[column sep=1.84ex]
		0\ar[r]&\Omega^{a}_{X/S}(\llog(D-D_1))\ar[r]\ar[d,"d"]&\Omega^{a}_{X/S}(\llog(D))\ar[r]\ar[d,"d"]&\Omega^{a-1}_{D_1/S}(\llog(D-D_1)|_{D_1})\ar[r]\ar[d,"d"]&0\\
		0\ar[r]&\Omega^{a+1}_{X/S}(\llog(D-D_1))\ar[r]&\Omega^{a+1}_{X/S}(\llog(D))\ar[r]&\Omega^{a}_{D_1/S}(\llog(D-D_1)|_{D_1})\ar[r]&0.
	\end{tikzcd}
	\] 
	After applying the snake lemma, we get the sequence
	\[
	\begin{aligned}
	0\to& Z\Omega_{X/S}^a(\llog(D-D_1))\to Z\Omega_{X/S}^a(\llog(D))\to Z\Omega^{a-1}_{D_1/S}(\llog(D-D_1)|_{D_1})\\
	 \to &B\Omega^{a+1}_{X/S}(\llog(D-D_1))\to \dots.
	\end{aligned}
	\]
	It is left to check that the boundary sequence induced by the snake lemma is the zero map.  Since the claim is \'etale local and $X\to S$ is log smooth, we may assume that $S$ is affine, $X\cong \mathbb{A}^m_S$ with coordinates $T_1,\dots, T_m$ and $D_i=V(T_i)$ for each $i=1,\dots ,n$.
	We write $\delta_i:= \dlog   T_i$ for $i=1,\dots, n$ and $\delta_i:= dT_i$ for $i=n+1,\dots, m$.
	
	Let $\tilde\eta$ be a section of  $Z\Omega_{D_1/S}^{a-1}(\log(D-D_1)|_{D_1})$ and let $\eta\wedge \dlog   T_1\in \Omega^{a}_{X/S}(\log(D))$ be a preimage. We may assume that $\eta$ is a sum of elements of the form
	\[
	g_{i_1,\dots, i_{a-1}} \delta_{i_1}\wedge \dots \wedge \delta_{i_{a-1}}, \quad i_1,\dots,i_{a-1}\neq 1, \; T_1\not| g_{i_1,\dots,i_{a-1}},
	\]
	where 
	\[
	g_{i_1,\dots, i_{a-1}}\in \mathcal{O}_X(X) =\mathcal{O}_S(S)[T_1,\dots, T_{a-1}].
	\]
	We may also assume that all monomials of the polynomials $g_{i_1,\dots, i_{a-1}}$ are not divisible by $T_1$, since such monomials are trivial when restricting to $D_1$.
	 
	Since $d\tilde\eta=0$, the differential $d(\eta\wedge \dlog   T_1)=d\eta \wedge \dlog   T_1$ lies in $\Omega^{a+1}_{X/S}(\llog (D-D_1))$. Hence $d\eta$ is necessarily divisible by $T_1$. But all monomials of $g_{i_1,\dots, i_{a-1}}$ are not divisible by $T_1$, implying that $d\eta=0$. 
\end{proof}

The vanishing result in \cref{AcyclicityProposition} is a key ingredient in the proof of the main theorem.
 \begin{proposition}[{Acyclicity of Blow-ups \cite[Theorem 6.1]{KelleyCohomology} }] \label{AcyclicityProposition}
 		Let $k$ be a perfect field of positive characteristic.
 		Let $X=\Spec(A)$ be a connected affine log smooth log scheme over $k$ whose log structure is defined by a normal crossing divisor $D\subset X$. Let $f:\bar{X}\to X=\Spec(A)$ be a blowup in a smooth center $Z\subset D$, such that the exceptional divisor $E$ intersects the strict transform $\bar{D}$ of the divisor $D$ transversally. Then \[H^i(\bar{X},\Omega^{n}_{\bar{X}/k}(\llog(E+\bar{D})))=0 \quad  \textrm{for } \quad i>0, n\geq 0.\]
 \end{proposition}

%% file: TameRZLet.tex
\subsection{Generalities about adic spaces}
\input{RZmorphisms}

\subsection{The Riemann-Zariski open site}
\input{TameRZ}
\subsection{Unramified sheaves}
\input{unramified}

%% file: RZmorphisms.tex
\subsubsection*{Discretely ringed adic spaces}
We will use adic spaces as defined in \cite[§1.1.]{HuberEtale}. An excellent overview on the subject may also be found in Wedhorn's book \cite{WedhornAdicSpaces}. We are only concerned with adic spaces which are locally of the form $\Spa(A,A^+)$, where $A$ has the discrete topology. We call such adic spaces \textit{discretely ringed}. 
\begin{definition}
Let $X\to S$ be a morphism of schemes. We define the set $\Spa(X,S)$ as the set of triples $P=(x,v,\epsilon)$ where $x$ is a point of $X$, $v$ an equivalence class of valuations on $k(x)$, and $\epsilon:\Spec(\mathcal{O}_v)\to S$ is a morphism fitting into the following diagram
\[
\begin{tikzcd}
	\Spec(k(x))\ar[r,"x"]\ar[d]&X\ar[d]\\
	\Spec(\mathcal{O}_v)\ar[r,"\epsilon"]&S.
\end{tikzcd}
\]
\end{definition}

For a given discretely ringed Huber pair $(A,A^+)$ 
\[
\Spa(\Spec(A),\Spec(A^+))\cong \Spa(A,A^+),
\]
by \cite[Lemma 3.1.2]{TemkinRZ}.

 We endow $\Spa(X,S)$ with the topology induced by open subsets of the form $\Spa(A,A^+)$ fitting into the diagram
 \[
 \begin{tikzcd}
 	\Spec(A)\ar[r,hook]\ar[d]& X\ar[d]\\
 	\Spec(A^+)\ar[r]&S
 \end{tikzcd}
 \]
 where $\Spec(A)\hookrightarrow X$ is an open immersion and $\Spec(A^+)\to S$ is of finite type and $(A,A^+)$ is a Huber pair. This also endows $\Spa(X,S)$ with the structure of an adic space. To elaborate, $\Spa(X,S)$ is the adic space $\mathcal{X}$ associated to the triple
 \[
 (\Spa(X,S),\mathcal{O}_{\mathcal{X}},(v)_{(x,v,\epsilon)})
 \]
 where $\mathcal{O}_{\mathcal{X}}(\Spa(A,A^+)):=A$.
 The assignment
 \[
 \Spa: \big\{\textrm{morphism of schemes} \big\}\to \big\{\textrm{Discretely ringed adic spaces}\big\}
 \]
 is functorial \cite[Lemma 2.1]{HTame}.

%% file: TameRZ.tex
 \subsubsection*{Horizontal specializations}
Let $\mathcal{X}$ be a discretely ringed adic space. We say that a point $y \in \mathcal{X}$ specializes to $x \in \mathcal{X}$ if every open neighborhood of $x$ contains $y$. We recall the definition of horizontal (or primary) specializations.
\begin{definition}[horizontal specialization]
	Let $\mathcal{X}=\Spa(A,A^+)$ be a discretely ringed  affinoid adic space. A point of $\mathcal{X}$ is a valuation $v:A\to \Gamma_{v}\cup \{0\}$ satisfying $v(a)\leq 1$ for all $a\in A^+$. Consider a convex subgroup $H\subset \Gamma_v$ which bounds $A$, i.e. for all $a\in A$, there exists $h\in H$ s.t. $v(a)\leq h$. One defines
	\[
	v|H:A\to H\cup \{0\}: a\mapsto \begin{cases}
		v(a) &\textrm{if }v(a)\in H, \\ 0 &\textrm{ otherwise.}
	\end{cases}
	\]
	This is a specialization of $v$. Such specializations are called \textit{horizontal} (or primary).
	
    A specialization $y \succ x$ in a discretely ringed adic space $\mathcal{Y}$ is horizontal if there exists an affinoid open set containing $x$ for which $y \succ x$ is horizontal.

    A point $x \in \mathcal{Y}$ is called \textit{Riemann-Zariski} (RZ) if it has no nontrivial horizontal specialization.
\end{definition}
The set of Riemann-Zariski points is called the Riemann-Zariski space. For a separated morphism $X\to S$ of qcqs schemes, we define the topological space
 \[\RZ(X,S)=\{\textrm{Riemann-Zariski points of } \Spa(X,S)\},\]
together with the underlying topology of $\Spa(X,S)$

\subsubsection*{Comparison between the open and Riemann-Zariski open sites}

 Consider the following morphisms
\[
\RZ(X,S)\xrightarrow{\iota}\Spa(X,S)\xrightarrow{\pi}\RZ(X,S).
\]
defined by Temkin \cite{TemkinRZ}. The first morphism is the inclusion, which is a topological embedding. The second morphism sends $x\in \Spa(X,S)$ to its RZ-point $x_\RZ$, which is the maximal horizontal specialization of $x$. The composition $\pi\circ \iota$ is the identity on $\RZ(X,S)$. In particular, $\RZ(X,S)$ is a retract of $\Spa(X,S)$.

The next proposition is featured in greater generality in \cite [Corollary 5.7]{HubnerBaseChange}
\begin{proposition}\label{OpenEqualRZ}
	Let $X\to S$ be a separated morphism of qcqs schemes. We have a natural isomorphism of cohomology groups 
	\[
	H^i(\Spa(X,S),\mathcal{F})\cong H^i(RZ(X,S),\pi_*\mathcal{F}).
	\]
\end{proposition}
\begin{proof}
	It suffices to check that $R\pi_*\mathcal{F}$ is concentrated in degree zero. To prove that we check that $\pi_*$ is equal to $\iota^*$, which is exact.
	
	Let $\mathcal{U}$ be an open subset of $\RZ(X,S)$. Then $\mathcal{U}= \pi^{-1}(\mathcal{U})\cap \RZ(X,S)$. Furthermore, $\pi^{-1}(\mathcal{U})$ is the smallest open subset of $\Spa(X,S)$ containing $\mathcal{U}$: indeed, it is the set of horizontal generalizations of points in $\mathcal{U}$, and every open subset containing $\mathcal{U}$ must contain all these generalizations. 
	
	By definition, $\iota^*\mathcal{F}$ is the sheafification of
	\[
	\mathcal{U}\mapsto \textrm{colim}_{\mathcal{V}} \mathcal{F}(\mathcal{V}),
	\]
	where the colimit is taken over $\mathcal{V}$ open in $\Spa(X,S)$, such that $\mathcal{U}\to \RZ(X,S)$ factors through $\mathcal{V}\cap \RZ(X,S)$. The open $\pi^{-1}(\mathcal{U})$ is initial for such open morphisms, hence 
	\[
	\textrm{colim}_{\mathcal{V}} \mathcal{F}(\mathcal{V})= \mathcal{F}(\pi^{-1}(\mathcal{U}))=\pi_*\mathcal{F}(\mathcal{U}).
	\] 
	It follows that $\iota^*\mathcal{F}=\pi_*\mathcal{F}$, which concludes the proof.
\end{proof}
Let $X\to S$ be a separated morphism of qcqs schemes. In \cite{TemkinRZ}, Temkin defines the Riemann-Zariski space as the topological space
\begin{equation}\label{RZmaps}
	\RZ(X,S)=\lim_{X\to\bar{Y}\to S}\bar{Y},
\end{equation}
where $\bar{Y}$ ranges over all $X$-\textit{modifications} $\bar{Y}$ of $S$: these are factorizations $X\xrightarrow{f}\bar{Y}\xrightarrow{g} S$ of $X\to S$, such that $f$ is schematically dominant and $g$  is proper. We often write $\bar{Y}$ instead of $X\xrightarrow{f}\bar{Y}\xrightarrow{g} S$. He then shows that the two descriptions agree since $X\to S$ is separated \cite[Corollary 3.4.7 and Theorem 3.5.1]{TemkinRZ}.
\begin{notation}\label{Notation|Y}
	Let $\bar{Y}$ be an $X$-modification of $S$. Let $\mathcal{F}$ be a sheaf on the topological space $\Spa(X,S)=\Spa(X,\bar{Y})$. We define the notation
\[		
\mathcal{F}|_{\bar{Y}}:= \Psi_{Y,*}\mathcal{F},
\]
	where $\Psi_Y:\Spa(X,S) \xrightarrow{\pi} \RZ(X,S)\to \bar{Y}$ is the natural morphism associated to the description $ \RZ(X,S)\cong \lim_{X\to\bar{Y}\to S}\bar{Y} $ in  (\refeq{RZmaps}).
\end{notation}
We apply \cite[Chapter 0, Corollary 3.1.20]{FujiwaraFoundationsRigidGeometry} \cref{OpenEqualRZ} to obtain the following corollary.
\begin{corollary}\label{CohomologyIsColimit}
	Let $X\to S$ be a separated morphism of qcqs schemes. We have a natural isomorphism
	\[H^i(\Spa(X,S),\mathcal{F})\cong \textrm{colim}_{\bar{Y}} H^i(\bar{Y},\mathcal{F}|_{\bar{Y}})
	\]
	where the colimit is taken over all $X$-modifications $\bar{Y}$ of $S$.
\end{corollary}
 \begin{proof}
 	By abstract nonsense \cite[Chapter 0, Corollary 3.1.20]{FujiwaraFoundationsRigidGeometry}, we have an isomorphism
	\[
	H^i(\RZ(X,S),\pi_*\mathcal{F}) \cong \textrm{colim}_{\bar{Y}} H^i(\bar{Y},\pi_*\mathcal{F}|_{\bar{Y}}).
	\]
	By \cref{OpenEqualRZ}, we have an isomorphism 
	\[
	H^i(\Spa(X,S),\mathcal{F})\cong H^i(\RZ(X,S),\pi_*\mathcal{F}).\qedhere
	\]
 \end{proof}

%% file: unramified.tex
In this section we define unramified sheaves, which relate sheaves on logarithmic schemes to sheaves on discretely ringed adic spaces. For more details see section 6 in \cite{LogDiff}.
\subsubsection*{The log site}
We fix, as usual, a perfect field $k$ of positive characteristic $p$. 

Consider a separated morphism of qcqs schemes $X\to S$ over $k$, where $X$ is smooth over $k$. A \textit{log smooth presentation} of $\mathcal{X}=\Spa(X,S)$ is a pair $(X,\bar{X})$ with $X$ open in $\bar{X}$, such that $\mathcal{X}=\Spa(X,\bar{X})$ and $\bar{X}$ is log smooth over $k$ when endowed with the compactifying log structure associated to $X\subset \bar{X}$. In particular, a discretely ringed adic space with a log smooth presentation is also smooth over $\Spa(k,k)$.

Hübner defines in \cite{LogDiff} the site $(X,S)_\llog$, whose objects are log smooth schemes $\bar{Y}$ such that we have the following commutative diagram
\[
\begin{tikzcd}
	Y\ar[r]\ar[d]&X\ar[d]\\
	\bar{Y}\ar[r]&S
\end{tikzcd}
\]
where $Y\to \bar{Y}$ is an open immersion, $\bar{Y}$ carries the compactifying log structure associated to $Y$ and $\Spa(Y,\bar{Y})\to \Spa(X,S)$ defines an open immersion. In particular $Y\to X$ is an open immersion of schemes and $\bar{Y}\to S$ is the normalization in $Y$ of a scheme of finite type over $S$.

Morphisms in $(X,S)_\llog$ are morphisms of log schemes. Coverings in $(X,S)_\llog$ are defined to be surjective families $(\bar{Y}_i\to \bar{Y})$ of (strict) open immersions.

\begin{example}
	The first example of interest are logarithmic differentials. Consider the presheaf $\Omega^{n,\log}$ defined by
	\[
	\bar{Y}\mapsto \Omega^{n,\llog}_{\bar{Y}}(\bar{Y}).
	\]
	This is a sheaf on $(X,S)_\llog$ since $\Omega^{n,\llog}_{\bar{Y}}$ is a Zariski sheaf on $\bar{Y}$.
\end{example}

\begin{definition}[{\cite[Definition 6.1]{LogDiff}}]
	An \textit{isomorphism in codimension one} in $(X,S)_\llog$ is a morphism $\bar{Y}'\to \bar{Y}$ of log schemes such that $\bar{Y}'_{\textrm{tr}}\cong \bar{Y}_{\textrm{tr}}\times_{\bar{Y}}\bar{Y}'$ and there exists an open subscheme $U\subset \bar{Y}$ containing all points of codimension $\leq 1$ in $\bar{Y}$ such that $U\times_{\bar{Y}}\bar{Y}'$ is isomorphic to $U$. 
	
	We write $\bar{Y}\sim_1\bar{Y}'$.
\end{definition}
We can now define unramified sheaves.
\begin{definition}[Unramified sheaves, {\cite[Definition 6.2]{LogDiff}}]
A sheaf $\mathcal{F}$ on $(X,S)_\llog$ is \textit{unramified} if for every open immersion $\bar{Y}'\to \bar{Y}$ in $(X,S)_\llog$ with dense image 
\[
\mathcal{F}(\bar{Y})\to \mathcal{F}(\bar{Y}')
\]
is injective, and for each isomorphism in codimension one  $\bar{Y}\sim_1\bar{Y}'$
\[
\mathcal{F}(\bar{Y})\to \mathcal{F}(\bar{Y}')
\]
is an isomorphism.
\end{definition}
\begin{example}[{\cite[Lemma 6.10]{LogDiff}}]
	\begin{itemize}
			\item The presheaf $\bar{\mathcal{O}}:\bar{Y}\mapsto \mathcal{O}_{\bar{Y}}(\bar{Y})$ is unramified.
			\item The sheaf of log differentials $\Omega^{n,\llog}$ is unramified.
			\item The sheaf of exact differentials $Z\Omega^{n,\llog}= \ker(\Omega^{n,\llog}\xrightarrow{d}\Omega^{n+1,\llog})$ is unramified. 
			
			\item The sheaf $B\Omega^{n,\llog}=d(\Omega^{n-1,\llog})$ is unramified. This will be proven later in \cref{BOmegaUnramified} using the Cartier operator.

	\end{itemize}
	More generally, any locally free $\bar{\mathcal{O}}$-module in $(X,S)_\llog$ is unramified.
\end{example}
\begin{lemma}\label{kernelOfUnramifiedIsUnramified}
	Consider the  exact sequence of sheaves in $(X,S)_\llog$
	\[
	0\to \mathcal{F}\to \mathcal{G} \xrightarrow{f}\mathcal{H}.
	\]
	If $\mathcal{G}$ and $\mathcal{H}$ are unramified, then the same is true for $\mathcal{F}$.
\end{lemma}
\begin{proof}
	\begin{itemize}
			\item \textbf{Injectivity: } This condition holds because $\mathcal{F}$ is a subsheaf of the unramified sheaf $\mathcal{G}$: Consider an open immersion $\bar{Y}'\to \bar{Y}$ with dense image. We have the diagram 
			\[
			\begin{tikzcd}
	 			\mathcal{F}(\bar{Y})\arrow[r]\arrow[d]&\mathcal{F}(\bar{Y}')\arrow[d]\\
				\mathcal{G}(\bar{Y})\arrow[r]&\mathcal{G}(\bar{Y}')
			\end{tikzcd}
			\]
			The lower horizontal and left vertical arrow are injections. Hence the upper horizontal arrow is also an injection.
			\item \textbf{Bijection for isomorphisms  in codimension one:} Let $\bar{Y}\sim_1 \bar{Y}'$ be an isomorphism in codimension one. We have the following diagram
			\[
			\begin{tikzcd}
				0\ar[r]&\mathcal{F}(\bar{Y}')\ar[r]\ar[d]&\mathcal{G}(\bar{Y}')\ar[r]\ar[d,"\cong"]&\mathcal{H}(\bar{Y}')\ar[d,"\cong"]\\
				0\ar[r]&\mathcal{F}(\bar{Y})\ar[r]&\mathcal{G}(\bar{Y})\ar[r]&\mathcal{H}(\bar{Y}).
			\end{tikzcd}
			\]
			Since $\mathcal{G}(\bar{Y}')\to \mathcal{G}(\bar{Y})$ and $\mathcal{F}(\bar{Y}')\to \mathcal{F}(\bar{Y})$ are isomorphisms we may apply the 5-lemma to obtain that $\mathcal{F}(\bar{Y}')\to \mathcal{F}(\bar{Y})$ is also an isomorphism. \qedhere
%
	\end{itemize}
\end{proof}
\subsubsection*{Induced sheaves on adic spaces}
An unramified sheaf $\mathcal{F}$ in $(X,S)_\llog$ induces a presheaf on adic spaces denoted $\mathcal{F}_{\lim}$. Before stating the definition of $\mathcal{F}_{\lim}$ we need the following lemma.
\begin{lemma}[{\cite[Lemma 6.3]{LogDiff}}]
	Let $\mathcal{F}$ be an unramified sheaf on $(X,S)_\llog$. If $\bar{Y}'\to \bar{Y}$ induces an isomorphism $\Spa(Y',\bar{Y}')\to \Spa(Y,\bar{Y})$, then the map $\mathcal{F}(\bar{Y})\to \mathcal{F}(\bar{Y}')$ is an isomorphism.
\end{lemma}
This enables us to associate a presheaf on $\Spa(X,S)$ to each unramified sheaf.
\begin{definition}
	Let $\mathcal{F}$ be an unramified sheaf on $(X,S)_\llog$. We define the presheaf $\mathcal{F}_{\lim}$ on $\Spa(X,S)$ as 
	$$\mathcal{F}_{\lim}(\mathcal{U})= \lim\limits_{\bar{Y}\in (X,S)_\llog/\mathcal{U}}\mathcal{F}(\bar{Y}),$$
	where $(X,S)_\llog/\mathcal{U}$ denotes all objects $\bar{Y}$ in $(X,S)_\llog$, such that $\Spa(\bar{Y}_\textrm{tr},\bar{Y})\to \Spa(X,S)$ factors through $\mathcal{U}$.
\end{definition}
Actually, $\mathcal{F}_{\lim}$ is not only a presheaf.
\begin{proposition}[{\cite[Proposition 8.9]{LogDiff}}]
	The presheaf $\mathcal{F}_{\lim}$ is a sheaf on the topological space $\Spa(X,S)$.
\end{proposition}
\begin{remark}
	Note that if $\mathcal{U}=\Spa(Y_{\textrm{tr}},\bar{Y})$ is associated to a log smooth log scheme $\bar{Y}$, then $(X,S)_{\log}/\mathcal{U}$ has a final object, namely $\bar{Y}$. Hence $\mathcal{F}_{\lim}(\mathcal{U})=\mathcal{F}(\bar{Y})$
	
	The operation $(-)_{\lim}$ taking unramified sheaves to sheaves on $\Spa(X,S)$
	is functorial and left exact. 
\end{remark}
We now can focus on our examples of interest
\begin{example}
	Let $\mathcal{X}=\Spa(X,S)$ be a discretely ringed adic space. Consider the sheaf $\Omega^{n,\llog}_{\lim}$. We have a differential 
		\[
		d:\Omega^{n,\llog}_{\lim}\to\Omega^{n+1,\llog}_{\lim},
		\]
		defined by  taking the limit on all differentials.
		By left exactness of the $\lim$ functor on unramified sheaves we have an isomorphism
		\[
		(Z\Omega^{n,\llog})_{\lim}\cong \ker(\Omega^{n,\llog}_{\lim}\xrightarrow{d}\Omega^{n+1,\llog}_{\lim}).
		\]
		For simplicity we will write $Z\Omega^{n,\llog}_{\lim}$ for $\ker(\Omega^{n,\llog}_{\lim}\xrightarrow{d}\Omega^{n+1,\llog}_{\lim})$.
\end{example}

%% file: OmegaPlus.tex
\subsection{Rangers and definition of \texorpdfstring{$\Omega^+$}{Omega⁺}}
\input{Rangers+}
\subsection{Logarithmic differentials on adic spaces}
\input{LogDiffAdic}
\subsection{The Cartier Operator}
\subsubsection*{The Cartier operator for logarithmic differentials}
\input{CartierOperatorLog}
\subsubsection*{Cartier operator on adic spaces}
\input{CartierOperatorAdic}

\subsection{Cohomology of \texorpdfstring{$\Omega^{+,n}$}{Omega⁺}}
\input{CohomologyofOmega+}

%% file: Rangers+.tex
As usual, let \( \mathcal{X} = \Spa(X, S) \) be a discretely ringed adic space over a perfect field $k$ of positive characteristic $p$. The sheaf \( \Omega^+ \) was introduced in \cite{LogDiff} and is defined by
\[
\Omega^+(\mathcal{U}) := \left\{ \eta \in \Omega(\mathcal{U}) \;\middle|\; |\eta|_x \leq 1 \text{ for all } x \in \mathcal{U} \right\}
\]
for a suitable definition of the seminorm
\[
|\cdot|_x : \Omega_{A_{\supp(x)}} \to \mathcal{R}_{\Gamma_x},
\]
known as the \emph{Kähler seminorm} (see below). This construction originates in Temkin’s paper \cite{TemkinMetrization}, where he defines a similar Kähler seminorm on the module of differentials \( \Omega_{K/k} \) for an extension of real-valued fields \( K/k \) as
\[
|\eta|_\Omega:= \inf_{\eta=\sum f_idg_i}\max_i\{|f_i|_{K}|g_i|_K\}.
\]
This definition, however, does not directly extend to local Huber pairs \( (A, A^+) \), as the value group \( \Gamma_A \) may lack infima. To address this, Hübner introduced the concept of \emph{rangers} in \cite[§5]{LogDiff}, which were further applied to the study of adic curves in recent work by Hübner and Temkin \cite[2.1.5]{HTCurves}.

\begin{definition}[{Rangers, \cite[§5]{LogDiff}, \cite[2.1.5]{HTCurves}}]
	Let \( \Gamma \) be an ordered abelian group (or more generally, an ordered set). A \emph{ranger} is a pair \( (\Gamma', \gamma') \), where \( \Gamma' \) is an ordered group (or ordered set) containing \( \Gamma \), and \( \gamma' \in \Gamma' \). Two rangers \( (\Gamma_1, \gamma_1) \) and \( (\Gamma_2, \gamma_2) \) are \emph{equivalent} if there exists a ranger \( (\Gamma', \gamma') \) such that \( \Gamma' \subseteq \Gamma_1 \cap \Gamma_2 \), and \( \gamma' \) maps to \( \gamma_1 \) in \( \Gamma_1 \) and to \( \gamma_2 \) in \( \Gamma_2 \).

	The set of equivalence classes of rangers is denoted \( \mathcal{R}_\Gamma \). It is  a  totally ordered set.

\end{definition}

Rangers are used to define the Kähler seminorm on modules of differentials over \textit{local Huber pairs}.  Let us shortly recall the definition of local Huber pairs: Let $\mX$ be an adic space and $x\in \mX$. The \textit{localization} of $\mX$ is the adic space
\[
\mX_x:= \lim\limits_{x\in \mU\subseteq \mX} \mU. 
\]

If $\mX$ is a discretely ringed affinoid $\Spa(A,A^+)$, then 
\[
	\Spa(A,A^+)_x=\Spa(B,B^+),
\]
where $B=A_{\supp(x)}$ and $B^+$ is the preimage of $k^+(x)\subset k(x)=A_{\supp(x)}/\mm_{\supp(x)} $ in $A_{\supp(x)}=B$. In particular, $\mX$ is an affinoid adic space. 

More generally, a Huber pair $(A,A^+)$ is called \textit{local} if $A$ is local with maximal ideal $\mm$  and $A^+$ is the preimage of a valuation ring $\mO \subseteq A/\mm$ in $A$. In particular, $A^+$ is a local ring with a valuation $|\cdot|$ and a maximal ideal
\[
\mm^+=\{a\in A|, |a|<1\}.
\]
Note also that $A^+=\{a\in A|,|a|\leq 1\}$ and $\mm=\{a\in A,|a|=0\}.$

The localization of an adic space is local. For more details on local Huber pairs, see  \cite[Section 10]{HSTame}.

\begin{definition}[{\cite[§7]{LogDiff}}]\label{def:ValuationOmega}
	Let \( (A, A^+) \) be a local Huber pair over a perfect field \( k \) of positive characteristic. The \emph{Kähler seminorm}
	\[
	|\cdot|:\Omega_{A/k}\to \mathcal{R}_\Gamma
	\]
	is defined by
	\[
	|\omega|:= \inf\limits_{\omega=\sum f_idg_i}\max\limits_i\big\{ |f_i|_A|g_i|_A\big\},
	\]
	where $\Gamma$ is value group associated to $(A,A^+)$.
\end{definition}

The usual Kähler differentials \( \Omega^n_{X/k} \) define a coherent \( \mathcal{O}_{\mathcal{X}} \)-module
\[
\Omega^{n}_{\mathcal{X}/k}:\Spa(A,A^+)\mapsto \Omega_{A/k}^n,
\]
 Using rangers. we can also define the sheaf $\Omega^{+,n}_{\mathcal{X}/k}$.
\begin{definition}
	
	Let \( \mathcal{X} \) be a discretely ringed adic space over a field \( k \) of positive characteristic. For any open subset \( \mathcal{U} \subseteq \mathcal{X} \), define
	\[
	\Omega^+_{\mathcal{X}/k}(\mathcal{U}) := \left\{ \eta \in \Omega_{\mathcal{X}/k}(\mathcal{U}) \;\middle|\; |\eta|_v|\leq 1 \text{ for all } v \in \mathcal{U} \right\}.
	\]
	Then set
	\[
	\Omega^{+,n}_{\mathcal{X}/k} := \bigwedge\nolimits^n_{\mO^+_\mX} \Omega^+_{\mathcal{X}/k}.
	\]
	
	We often omit the subscript and simply write \( \Omega^{+,n} \) when the context is clear.
	\begin{definition}[Exact and closed differentials]
		We define the sheaf of \emph{exact differentials} as
		\[
		B\Omega^{+,n} := d(\Omega^{+,n-1}),
		\]
		and the sheaf of \emph{closed differentials} as
		\[
	Z\Omega^{+,n} := \ker(\Omega^{+,n} \xrightarrow{d} \Omega^{+,n+1}).
		\]
		
	\end{definition}
\end{definition}
The next property enables us to use $\Omega^{+,n}$ in the context of tame cohomology.
\begin{proposition}[{\cite[Proposition 7.9]{LogDiff}}]
	The presheaf $\Omega^{+,n}$ is a sheaf in the tame and strongly \'etale topology.
\end{proposition}
We also define a valuation on $\Omega^n_{A}$ with values in $\mathcal{R}_\Gamma$, generalizing \cref{def:ValuationOmega}.
\begin{definition}
	Let $(A,A^+)$ be a local Huber pair over a field $k$ of positive characteristic. Let $\Gamma$ be the value group associated to $(A,A^+)$. Let $n$ be an integer. The \emph{Kähler seminorm} on $\Omega_{A/k}^n$ is the map
	\[
	|\cdot|: \Omega_A^n\to \mathcal{R}_\Gamma
	\]
	defined by 
	\[
	|\omega|:= \inf\limits_{\omega=\sum f_idg_{1,i}\wedge \dots\wedge dg_{n,i}} \max\limits_i\big\{|f_i \cdot g_{1,i}\cdots g_{n,i}| \big\}.
	\]
\end{definition}
\begin{proposition}\label{prop:EquivalentDescriptionsOmega+n}
	For every integer $n$, we have the equality 
	\begin{equation}\label{eq:twoDefOmega+n}
			\Omega_{(A,A^+)}^{+,n}=\big\{ \omega \in \Omega_A^n ; |\omega|\leq 1 \big\}
	\end{equation}
\end{proposition}
\begin{proof}
	Let $\omega$ be a differential in $\Omega_{(A,A^+)}^{+,n}$. Then $\omega$ is a sum  of wedge products of elements of the form $fdg$ satisfying $|fg|\leq 1$ \cite[Lemma 7.3.]{LogDiff}(cf. \cref{lem:plusisLog}). Therefore $\omega=\sum\limits_{i}f_idg_{1,i}\wedge \dots \wedge dg_{n,i}$ such that $|f_ig_{1,i}\dots g_{n,i}|\leq 1$. Hence $\omega$ is in the RHS of  \eqref{eq:twoDefOmega+n}.
	
	Consider $\omega \in \Omega_A^n$ lying in the RHS of \eqref{eq:twoDefOmega+n}. According to \cite[Lemma 7.2]{LogDiff}, one may express $\omega$ as a sum $\sum\limits_{i}f_idg_{1,i}\wedge \dots \wedge dg_{n,i}$ such that $|f_ig_{1,i}\dots g_{n,i}|\leq 1$. Thus the proof of \cref{prop:EquivalentDescriptionsOmega+n} is reduced to the following claim.
	\begin{claim}\label{claim:EquivalentDescriptionsOmega+n}
		Let $f,g_1,\dots,g_n$ be elements in $A$ such that $|fg_1\dots g_n|\leq 1$, then $\eta:=fdg_1\wedge\cdots \wedge g_n\in \Omega_{(A,A^+)}^{+,n}$.
	\end{claim}
	We prove \cref{claim:EquivalentDescriptionsOmega+n} inductively. When $n=1$ this is true by definition. Assume that the claim is true for $n-1$. 
	
	\textrm{Case 1:} $|g_1|=\dots=|g_n|=0$. Then $fdg_1\wedge \dots \wedge d_{n-1}\in \Omega_{(A,A^+)}^{+,n-1}$ and $dg_n\in \Omega_{(A,A^+)}^{+}$. Hence $\eta\in\Omega_{(A,A^+)}^{+,n}$.
	
	\textrm{Case 2:} $|fg_1\dots g_{n-1}|=0$ and $|g_{n}|\neq 0$. Then express $\eta$ as the wedge product 
	\[
	(fg_n dg_1\wedge \dots\wedge dg_{n-1} )\wedge  (dg_n/g_n).
	\]
	Then $fg_ndg_1\wedge\dots \wedge dg_{n-1}\in \Omega_{(A,A^+)}^{+,n-1}$ due to the induction step and $dg_n/g_n\in \Omega_{(A,A^+)}^{+}$ by definition.
	
	\textrm{Case 3:} We may now assume, without loss of generality, that $|f|, |g_1|,\dots,|g_n| \neq 0$. Then 
	\[
	\eta= (fg_1\dots g_n)(dg_1/g_1)\wedge \dots \wedge (dg_n/g_n),
	\]
	lies in $\Omega_{(A,A^+)}^{+,n}$. 
\end{proof}

%% file: LogDiffAdic.tex
For every morphism of Huber pairs $(B,B^+)\to(A,A^+)$ we can define the associated logarithmic differential
\[
\Omega^{\textrm{log}}(\Spa(A,A^+))=\Omega_{(A,A^+)/(B,B^+)}^{\textrm{log}}=\Omega^{\llog}_{(A^+\cap A^\times\to A)/(B^+\cap B^\times\to B)}.
\]
In this article we will mostly consider $(B,B^+)=(k,k)$ for a  perfect field $k$ of characteristic $p>0$. 

We also define $\Omega^{n,\llog}:=\bigwedge^n\Omega^{\llog}$.
This is however not a sheaf, as shown in \cite[Example 4.4.]{LogDiff}.
\begin{lemma}[{\cite[Lemma 7.3.]{LogDiff}}]\label{lem:plusisLog}
	Let $k$ be a perfect field of characteristic $p$. Let $(A,A^+)$ be a local Huber pair.
	We have 
	\[
	\Omega^{n,\log}_{(A,A^+)}\cong \Omega^{n,+}_{(A,A^+)}.
	\]
\end{lemma}

This leads to the proposition:
\begin{proposition}[{\cite[Proposition 7.11]{LogDiff}}]
	Let $\mathcal{X}$ be a discretely ringed adic space over a perfect field $k$ of characteristic $p$. Then $\Omega^{+,n}$ is the sheafification of $\Omega^{n,\llog}$.
\end{proposition}

\begin{remark}
	Under the same assumptions, we have the following inclusion 
	\[
	\Omega^{+,n}_{\mathcal{X}/k}\subseteq \Omega^{n,\llog}_{\mathcal{X}/k,\lim} \subseteq \Omega^{n}_{\mathcal{X}/k},
	\]
	as $\Omega^{n,\log}_{\mathcal{X}/k,\lim}$ is a sheaf that contains $\Omega^{n,\log}_{\mathcal{X}/k}$.
\end{remark}

If $\mathcal{X}=\Spa(X,\bar{X})$ is associated to a log smooth pair $(X,\bar{X})$, then
\[
\Omega^{+,n}_{\mathcal{X}/k}(\Spa(X,\bar{X}))\to  \Omega^{n,\llog}_{\mathcal{X}/k,\lim}(\Spa(X,\bar{X}))\to \Omega^{n,\llog}_{\bar{X}/k}(\llog \bar{X}\setminus X).
\]
is an isomorphism. This is \cite[Theorem 8.12]{LogDiff}.
\begin{conjecture}\label{Hypothesis_L}
	Given a smooth discretely ringed adic space $\mathcal{X}=\Spa(X,S)$ over $k$, we conjecture that 
	\[
	\Omega_{\mathcal{X},\lim}^{n,\llog}\cong \Omega^{n,+}_{\mathcal{X}} 
	\] 	
	for every $n\in \mathbb{N}$. 
\end{conjecture}

\begin{proposition}\label{Hypothesis_L_Resolution}
	Under the assumption of resolution of singularities (namely \cref{ConjResolution3} and \cref{ConjResolution2}) over $k$ ) \cref{Hypothesis_L} holds. 
\end{proposition}
\begin{proof}
	Under the given assumptions, every open in $\mathcal{X}$ is covered by adic spaces of the form $\mathcal{U}=\Spa(U,\bar{U})$, for which the pair $(U,\bar{U})$ is a log smooth presentation of $\mathcal{U}$. Hence over this covering $\Omega^{+,n}$ and $\Omega^{n,\llog}_{\lim}$ agree as subsheaves of $\Omega^{n}$. The result follows.
\end{proof}
\begin{corollary}
	Let $\mathcal{X}=\Spa(X,S)$ be a smooth discretely ringed adic space over $k$. Under the assumption of resolution of singularities over $k$, we have $Z\Omega^{+,n}_{\mathcal{X}/k}\cong Z\Omega^{n,\llog}_{\mathcal{X}/k,\lim}$ and $B\Omega^{+,n}_{\mathcal{X}/k}\cong B\Omega^{n,\llog}_{\mathcal{X}/k,\lim}$ for every $n\in \mathbb{N}$.
\end{corollary}
\begin{proof}
	Under the assumption of resolution of singularities, we have $\Omega^{+,n}_{\mathcal{X}/k}\cong \Omega^{n,\llog}_{\mathcal{X}/k,\lim}$ (\cref{Hypothesis_L_Resolution}). The result follows by taking kernels and images of the $d$ map, and using the fact that $d$ is compatible with the isomorphism $\Omega^{+,n}_{\mathcal{X}/k}\cong \Omega^{n,\llog}_{\mathcal{X}/k,\lim}$ on opens of the form $\Spa(U,\bar{U})$ where $(U,\bar{U})$ is a log smooth presentation.
\end{proof}\newpage

%% file: CartierOperatorLog.tex
Let $k$ be a perfect field of positive characteristic. We write $S=\Spec(k)$.

Let $\mathcal{X}=\Spa(X,T)$ be a smooth discretely ringed adic space over $k$. Assume that $\mathcal{X}=\Spa(X,\bar{X})$ where $\bar{X}$ is a log smooth log scheme with the compactifying log structure associated to the open subscheme $X$. Denote the open immersion $X\to \bar{X}$ by $j$.

We start by introducing some notation. Consider the absolue Frobenius endomorphism
\[
F_{\textrm{abs}}: X\to X
\] 
induced by sending a section $\alpha$ in $\mathcal{O}_X$ to $\alpha^p$. We denote by $X^{(p)}$ the pullback
\[
\begin{tikzcd}
	X^{(p)}\ar[r,"W"]\arrow[d,"\pi^{(p)}"]&X\ar[d,"\pi"]\\
	S\ar[r,"F_{\textrm{abs}}"]&S
\end{tikzcd}
\]

The morphism $W$ induces an isomorphism on $X^{(p)}\to X$. The structure morphism $\pi^{(p)}:X^{(p)}\to S$ is associated to the morphism $\mathcal{O}_S\to \mathcal{O}_{X^{(p)}}(\cong\mathcal{O}_X):f\mapsto f^{\frac{1}{p}}$. On the level of topological spaces, the morphism $W$ is just the identity.

The absolute Frobenius $F_{\textrm{abs}}:X\to X$ induces the relative Frobenius $F:X\to X^{(p)}$. 

Now note that $\Omega_{X^{(p)}/S}=F_*\Omega_{X/S}$. Similarly, if $D\subseteq X$ is a normal crossing divisor, then $\Omega_{X^{(p)}/S}(\llog D^{(p)})=F_*\Omega_{X/S}(\llog D)$.

Let us first recall a theorem of Cartier, see \cite[Theorem 7.2]{KatzNilpotent}.
\begin{proposition}[Inverse Cartier operator]
	Assume that $X$ is smooth over the prefect field $k$. 
	
	There exists for each $n\in \mathbb{N}$ an isomorphism 
	\[
	\bar{C}^{-1}:\Omega^n_{X^{(p)}/S}\to F_*\big(Z\Omega^n_{X/S}/B\Omega^n_{X/S}\big)
	\]
	of $\mathcal{O}_{X^{(p)}}$-modules, which satisfies
	\begin{itemize}
		\item $\bar{C}^{-1}(1)=1$ when $n=0$.
		\item $\bar{C}^{-1}(df)=[f^{p-1}df]$ when $n=1$.
		\item $\bar{C}^{-1}(\omega \wedge \eta) = \bar{C}^{-1}(\omega)\wedge \bar{C}^{-1}(\eta)$.
	\end{itemize}
	If on the other hand $\bar{X}$ has the structure of a log smooth log scheme then the same holds for the log differentials:
	\[
	\bar{C}^{-1}: \Omega^{n,\llog}_{\bar{X}^{(p)}/S}\to F_*(Z\Omega^{n,\llog}_{\bar{X}/S}/B\Omega^{n,\log}_{\bar{X}/S}).
	\]
	This isomorphism is compatible with 
	\[
	j_*\bar{C}^{-1}: j_*\Omega^n_{X^{(p)}/S}\to j_*F_*(Z\Omega^n_{X/S}/B\Omega^n_{X/S}),
	\]
	where $j_*:X\to \bar{X}$ is the open immersion associated to the trivializing subscheme $X$ of $X$.
	\end{proposition}
	We apply the fact that $\Omega_{X^{(p)}/S}^{n,\llog}$ is isomorphic to $F_*\Omega_{X/S}^{n,\llog}$ as an $\mathcal{O}_{X^{(p)}}$-module, and hence as a group, to directly obtain the following proposition.
\begin{proposition}\label{CartierOnLog}
	There exists a unique family of additive maps $C:Z\Omega_{\bar{X}/k}^{n,\llog}\to\Omega_{\bar{X}/k}^{n,\llog}$ with the properties 
	\begin{itemize}
		\item $C(1)=1$ in $\mathcal{O}_{\bar{X}}=\Omega_{\bar{X}/k}^{0,\llog}$,
		\item $C(f^p \omega)=fC(\omega)$ for each $f\in \Omega_X, \omega \in Z\Omega_{X/k}^{r,\llog}$,
		\item $C(\omega\wedge \omega')=C(\omega)\wedge C(\omega')$ for closed $\omega$ and $\omega'$ in $Z\Omega_{\bar{X}/k}^{r,\llog}$,
		\item $C(\omega)=0$ if and only if $\omega=d\eta$, where $\eta$ is a logarithmic differential,
		\item $C(f^{p-1}df)=df$.
	\end{itemize}
	  Furthermore, this is compatible with the pushforward of the Cartier operator
	  \[
	  j_*C: j_*Z\Omega_{X/k}\to j_*\Omega_{X/k}.
	  \]
\end{proposition}
\begin{proof}
	This follows directly from the previous proposition. The $\mathcal{O}_{X{(p)}}$-linearity translates to the second condition. The fact that $\bar{C}^{-1}$ is an isomorphism corresponds to the fourth condition. The uniqueness of $\bar{C}^{-1}$ is equivalent to that of $C$.
\end{proof}
The next proposition is a motivation for the use of logarithmic differentials.
\begin{proposition}[{\cite[Lemma 2.4.6]{sato2007logarithmic}}]\label{C-1ShortExactLog}
	Let $X$ be a log smooth log scheme over a perfect field $k$. The sequence 
	\[
	0\to j_*\nu_{X_{\textrm{tr}}}(r)\to Z\Omega^{r,\log}_{X/k}\xrightarrow{C-1}\Omega_{X/k}^{r,\log}\to 0
	\]
	is exact in  $\Sh(X_\et)$. The sheaf $j_*\nu_{X_{\textrm{tr}}}(r)$ is the pushforward of $\nu_{X_{\textrm{tr}}}$ via $j: X_\textrm{tr}\to X$. 	
\end{proposition}

%% file: CartierOperatorAdic.tex
In this section we study the properties of the Cartier operator on adic spaces. Consider a discretely ringed adic space $\mathcal{X}$ over a perfect field $k$.
Recall that we have the inclusions 
\[
\Omega_{\mathcal{X}}^{+,n}\subseteq \Omega_{\mathcal{X},\lim}^{n,\llog}\subseteq \Omega^n_{\mathcal{X}}.
\]
To prove some of the needed properties of the Cartier operator we sometimes need to work with $\Omega^{+}_{\mathcal{X}}$ and sometimes with $\Omega^{n,\llog}_{\mathcal{X},\lim}$. The existence of a Cartier operator may first be checked on the site $(X,S)_\llog$.

\begin{proposition}
	The Cartier operator defined in \cref{CartierOnLog} induces a surjective morphism 
	\[
	C: Z\Omega^{n,\llog}\to \Omega^{n,\llog}
	\]
	of sheaves on $(X,S)_\llog$. Its kernel is $B\Omega^{n,\llog}$.
\end{proposition}
\begin{proof}
	This is a direct consequence of \cref{CartierOnLog}. 
\end{proof}

\begin{corollary}\label{BOmegaUnramified}
	The sheaf $B\Omega^{n,\llog}$ is unramified.  
\end{corollary}
\begin{proof}
	This follows from the fact that the $B\Omega^{n,\llog}$ is the kernel of the surjective morphism $C$ between unramified sheaves and from \cref{kernelOfUnramifiedIsUnramified}.
\end{proof}
Applying the functor $(\cdot)_{\lim}$ yields the following Proposition.
\begin{proposition}\label{ShortExactOnClog}
	There exists an exact sequence of sheaves on $\Spa(X,S)$
\[
0\to B\Omega_{\mathcal{X}/k,\lim}^{n,\llog}\to Z\Omega_{\mathcal{X}/k,\lim}^{n,\llog}\xrightarrow{C}\Omega_{\mathcal{X}/k,\lim}^{n,\llog}.
\]
\end{proposition}
\begin{corollary}\label{ShortExactOnC}
	Under the assumption of \cref{Hypothesis_L} we have a short exact sequence of sheaves on $\mathcal{X}:=\Spa(X,S)$
	\[
	0\to B\Omega_{\mathcal{X}/k}^{+,n}\to Z\Omega_{\mathcal{X}/k}^{+,n}\xrightarrow{C}\Omega_{\mathcal{X}/k}^{+,n}\to 0.
	\]
\end{corollary}
\begin{proof}
	Since we assumed that $\Omega^{+,n}_{\mathcal{X}/k}\cong \Omega^{n,\log}_{\mathcal{X}/k,\lim}$ by \cref{Hypothesis_L}, the exactness of the LHS follows from \cref{ShortExactOnClog}. 
	
	The exactness of the RHS is checked on the level of stalks. Let $x\in \mathcal{X}$ and $(A,A^+)$ be the local Huber pairs associated to $x$. Let $\eta\in \Omega_{\mathcal{X}/k,x}^{+,n}\cong\Omega^{+,n}_{(A,A^+)/k}$. We can express $\eta$ as a sum 
	\[
	\eta = \sum_{i}f_i dg_{i,1}\wedge \dots \wedge dg_{i,n},
	\] 
	such that $f_i,g_{i,1},\dots ,g_{i,n}$ lie in $A$ and $|f_i\cdot g_{i,1} \dots  g_{i,n}|_x\leq 1$. Then 
	\[
	\gamma:= \sum_{i}f_i^{p}(g_{i,1}\cdots g_{i,n})^{p-1}dg_{i,1}\wedge \dots \wedge dg_{i,n},
	\]
	is a preimage of $\eta$ under $Z\Omega_{A/k}^{n}\xrightarrow{C}\Omega_{A/k}^{n}$. Since
	\[
	|(f_i\cdot g_{i,1} \dots  g_{i,n})^p|_x=|f_i\cdot g_{i,1} \dots  g_{i,n}|^p_x\leq 1,
	\]
	we deduce that $\gamma$ lies in $Z\Omega^{+,n}_{(A,A^+)/k}\cong Z\Omega^{+,n}_{\mathcal{X}/k,x}$.
	
	In conclusion, the Cartier operator $Z\Omega_{\mathcal{X}/k}^{+,n}\xrightarrow{C}\Omega_{\mathcal{X}/k}^{+,n}$ is surjective.
\end{proof}
\begin{definition}\label{Def:LogDeRhamWitt}
	The log de Rham-Witt sheaf $\nu_{\mathcal{X}}(n)$ is defined as the \'etale additive subsheaf of $\Omega_{\mathcal{X}}^{n}$ whose sections are \'etale locally generated by wedge products over $\mathbb{F}_p$ of the form
	\[
	df/f, \quad f\in \mathcal{O}_{\mathcal{X}}^\times.
	\]
	Alternatively, $\nu_{\mathcal{X}}(n)$ is the kernel of $C-1:Z\Omega_{\mathcal{X}/k}^n\to \Omega_{\mathcal{X}/k}^n$. 
	
	The latter definition is stated just after Remark 1.2. in \cite{MilneDuality}. The fact that it is equivalent to the first definition is stated in \cite[Lemma 1.6]{MilneValuesOfZeta}.
\end{definition}
In \cite{MilneValuesOfZeta}, one also defines the sheaves of $\mathbb{Z}/p^{m}\mathbb{Z}$-modules $\nu_{m}(n)$ on a given scheme $X$ over a perfect field $k$ of characteristic $p$. These are the subsheaves of the de Rham Witt sheaves $W_m\Omega^n_{X/k}$ generated by $\dlog \underline{f_1 }\wedge \dots \wedge \dlog \underline{f_n}$ (see the paragraph after Lemma 1.6 in \cite{MilneValuesOfZeta}).
\begin{definition}
	Over a discretely ringed adic space $\mathcal{X}$ over $k$ we define the sheaves
	\[
	\nu_{m,\mX}(n):\Op(\Spa(X,S))\to \mathbb{Z}/p^m\mathbb{Z}-\textrm{modules}: \Spa(U,T)\mapsto \nu_{m,\mX}(n)(U).
	\]
	These are \'etale sheaves and hence by restriction sheaves in the tame topology. We drop the subscript $\mX$ when the context is clear. We may also drop the index $m$ if $m=1$.
\end{definition}
\begin{remark}
	The sheaves $\nu_{r}(n)$ are Zariski-locally generated sub-$\mathbb{Z}/p^r\mathbb{Z}$-sheaves of the de Rham-Witt sheaves $W_r\Omega_\mX^n$ which are generated by $\dlog \underline{f_1} \wedge \dots \wedge \dlog \underline{f_n}$. This is \cite[Theorem 1.2]{morrowLogHodgeWitt}  .
\end{remark}
The next lemma is needed in order to generalize the purity theorem to $\mathbb{Z}/p^m\mathbb{Z}$-modules for any $m$.
\begin{lemma}\label{highermSES}
	We have a short exact sequence of sheaves on the topological space of $\mathcal{X}$
	\[
	0\to \nu_{m}(n)\to \nu_{m+m'}(n)\to \nu_{m'}(n)\to 0.
	\]
\end{lemma}
\begin{proof}
	This is a direct consequence of the same result on schemes proven in \cite[Lemma 1.7]{MilneValuesOfZeta}.
\end{proof}

\begin{proposition}
	 Assume \cref{Hypothesis_L}. Then the morphism 
	\[
	Z\Omega^{+,n}_{\mathcal{X}}\xrightarrow{C-1} \Omega^{+,n}_{\mathcal{X}}
	\]
	is surjective in the tame and strongly \'etale topology over $\mathcal{X}$.
\end{proposition}
\begin{proof}
	Let us check the surjectivity of $C-1$. This is checked on stalks. Assume we have a discretely ringed local Huber pair $(A,A^+)$. Denote by $|\cdot|:A\to \Gamma$ the valuation associated to $(A,A^+)$. Consider an element
	 $$\eta= h \dlog f_1\wedge \dots \wedge  \dlog f_n \in \Omega_{(A,A^+)}^{+,n},$$
	 where $h$ lies in $A^+$.
	 
	 A preimage of this in $Z\Omega^n_A$ under $C-1$ is of the form 
	$$\eta'=\gamma^p \dlog f_1 \wedge \dots \wedge \dlog f_n,$$
	with $\gamma \in A^+$ satisfying $\gamma^p -\gamma = h$. Such a $\gamma$ exists since 
	$$\Spec A^+[T]/(T^p-T-h)$$
	is \'etale and hence induces a finite strongly \'etale morphism of adic spaces
	\[
	\Spa(A[T]/(T^p-T-h), A^+[T]/(T^p-T-h))\to \Spa(A,A^+).
	\] 
	It is left to check that $\eta'$ lies in $\Omega^+_{(A,A^+)}$. To prove this, we need to check that $|\gamma^p|\leq 1$. Assume that $|\gamma|>1$. Then $|\gamma^p|=|\gamma|^p>|\gamma|$. Hence $1<|\gamma|^p=|\gamma^p-\gamma|=|h|\leq 1$, which is a contradiction.
	
	It follows that $|\gamma^p|\leq 1$. Therefore $\eta'$ lies in $Z\Omega_{(A,A^+)}^{+,n}$
\end{proof}
\begin{theorem}\label{C-1Shortexact}
	The sheaf $\nu(n)$ is the kernel of 
	\[
	Z\Omega_{\mathcal{X},\lim}^{n,\llog}\xrightarrow{C-1}\Omega_{\mathcal{X},\lim}^{n,\llog}.
	\]
	In particular, if \cref{Hypothesis_L} holds, then we have a short exact sequence
	\[
		0\to\nu(n)\to Z\Omega^{+,n}_{\mathcal{X}}\xrightarrow{C-1} \Omega^{+,n}_{\mathcal{X}}\to 0,
	\]
	which further implies that $\nu(n)$ is a sheaf in the tame topology.
\end{theorem}
\begin{proof}
Note that $\nu(n)$ is the kernel of $Z\Omega_{\mathcal{X}}^n\xrightarrow{C-1}\Omega_{\mathcal{X}}^n$. Furthermore $\nu(n)$ is the subsheaf of $Z\Omega_{\mathcal{X}}^n$ generated by sums of wedge products of logarithmic differentials, i.e. elements of the form $\dlog f=df/f$, for $f$ a section in $\mathcal{O}_X$. These elements have valuation $\leq 1$ for any point $x\in \mathcal{X}$ since $|df|_x\leq |f|_x$ by definition.
	
	But we already know that $Z\Omega^{+,n}_{\mathcal{X}}\subseteq Z\Omega^{n,\llog}_{\mathcal{X},\lim}$ with equality under the assumption of \cref{Hypothesis_L}.
\end{proof}

\begin{remark}
	\begin{itemize}
		\item An alternative less conceptual proof of \cref{C-1Shortexact} uses \cref{C-1ShortExactLog} together with assumption of resolution of singularities.
		\item We will present in a future article an alternative proof of the existence of the Cartier operator $C:Z\Omega_\mX^{+,n}\to \Omega_\mX^{+,n}$ without the need of resolution of singularities or \cref{Hypothesis_L}. The proof relies on a similar result of Gabber presented in the appendix of \cite{KSTVoerst}.
	\end{itemize}

\end{remark}

%% file: CohomologyofOmega+.tex
In this subsection we will check that the tame and the strongly \'etale cohomology of $\Omega^{+,n}_{\mathcal{X}/k}$ agree with the Riemann-Zariski open cohomology. Before doing so, we recall the definitions of the tame and strongly \'etale topology on adic spaces. 
 \subsubsection*{Tame and strongly \'etale topology}
 We assign to an adic space $\mathcal{X}$ the tame site $\mathcal{X}_t$ and the strongly \'etale site $\mathcal{X}_\textrm{s\'et}$ as defined in \cite{HTame}. We recall their definitions.
 \begin{definition}[{\cite[Definition 3.3.]{HTame}}]
	Let $L/K$ be a finite separable extension of valued fields, whose valuation rings are $\mathcal{O}_L$ and $\mathcal{O}_K$. Let $p$ be the residue characteristic of $L$ and $K$. We say that $L/K$ is \textit{unramified} if the extension $\mathcal{O}_L$ is the localization of an \'etale $\mathcal{O}_K$-algebra. It is said to be \textit{ramified} if it is not unramified.
	
	The extension $L/K$ is called \textit{tamely ramified} if $L^{\textrm{sh}}/K^{\textrm{sh}}$ is a prime-to-$p$ extension, where $L^{\textrm{sh}}$ and $K^{\textrm{sh}}$ are the strict henselizations of $L$ and $K$ respectively.  
	
	An \'etale morphism $f:\mathcal{Y}\to \mathcal{X}$ of adic space is \textit{tame} (resp. \textit{strongly}) \'etale if for every $x\in \mathcal{Y}$, the extension of valued fields $k(x)/k(f(x))$ is tamely ramified (resp. unramified).
	
	The tame and strongly \'etale coverings are \'etale coverings $\{\mathcal{X}_i\to \mathcal{X}\}$ whose morphisms are tame or strongly \'etale, respectively. These coverings induce the tame and strongly \'etale sites, denoted by $\mathcal{X}_{\tame}$ and $\mathcal{X}_{\sett}$, respectively.
 \end{definition}
 
We then have the chain of morphisms of sites
 \[
 \mathcal{X}_\et\xrightarrow{\phi} \mathcal{X}_{\tame}\xrightarrow{\psi} \mathcal{X}_{\sett}.
 \]
Assume that $\mathcal{X}$ is defined over $\mathbb{F}_p$. Given a $p$-torsion sheaf $\mathcal{F}$ on the tame site, we have
\[
H^i(\mathcal{X}_\tame, \mathcal{F}) \cong H^i(\mathcal{X}_\sett, \psi_*\mathcal{F}), \quad \forall i \geq 0,
\]
by \cite[Proposition 8.5]{HTame}.  Likewise, for a torsion sheaf $\mathcal{G}$ on the \'etale site with torsion prime to $p$, we have
 \[
 H^i(\mathcal{X}_\et, \mathcal{G}) \cong H^i(\mathcal{X}_\tame, \phi_* \mathcal{G}), \quad \forall i \geq 0,
 \]
 by \cite[Proposition 8.2]{HTame}.

Given a closed adic subspace $i:\mathcal{Z}\to \mathcal{X}$ we denote by $j:\mathcal{U}=\mathcal{X}\backslash \mathcal{Z}\to \mathcal{X}$ the natural open immersion. For $\tau\in \{\sett, \tame\}$, define 
\[
\begin{aligned}
	&i^!:&\{\textrm{abelian sheaves over }\mathcal{X}_\tau\}\to &\{\textrm{abelian sheaves on }\mathcal{X}_\tau \textrm{ with support on } \mathcal{Z}\}\\
	&&\mathcal{F}\mapsto &\ker (\mathcal{F}\to j_*j^*\mathcal{F}).
\end{aligned}
\]
Furthermore one defines the cohomology with support 
\[
H^n_{\mathcal{Z},\tau}(\mathcal{X},-)=R^n(\Gamma(\mathcal{X},i^!(-)).
\]
This fits into the long exact sequence of excision:
\begin{equation}\label{eq:Excision}
	\cdots \to H^{n-1}_\tau(\mathcal{U}, \mathcal{F}) \to H^n_{\mathcal{Z},\tau}(\mathcal{X}, \mathcal{F}) \to H^n_\tau(\mathcal{X}, \mathcal{F}) \to H^n_\tau(\mathcal{U}, \mathcal{F}) \to \cdots
\end{equation}

\begin{remark}\label{rmk:tamevsstronglyet}
	Using Propositions 8.2 and 8.5 of \cite{HTame} and the five lemma, one obtains analogous isomorphisms for cohomology with support
	 \[
	H^i_{\mathcal{Z}}(\mathcal{X}_\tame,\mathcal{F})\cong H^i_{\mathcal{Z}}(\mathcal{X}_\sett,\psi_*\mathcal{F}), \quad \forall i\geq 0
	\]  
	when $\mathcal{F}$ is a $p$-torsion sheaf on $\mathcal{X}_\tame$. Similarly, for torsion sheaves $\mathcal{G}$ on $\mathcal{X}_\et$ with torsion prime to $p$, one has
	\[
	H^i_{\mathcal{Z}}(\mathcal{X}_\et,\mathcal{G})\cong H^i_{\mathcal{Z}}(\mathcal{X}_\tame, \phi_* \mathcal{G}), \quad \forall i\geq 0.
	\]
\end{remark}
\subsubsection*{Comparision with the Riemann-Zariski cohomology	} 
For simplicity we write $\Omega^{+,n}$ instead of $\Omega^{+,n}_{\mathcal{X}/k}$. We similarly drop the subscript on $\Omega^{n,\log}$.
\begin{proposition}\label{StronglyEtaleEqualOpen}
	Let $\mathcal{X}$ be a discretely ringed adic space. Then 
	\[
	H^i_{\textrm{s\'et}}(\mathcal{X},\Omega^{+,n})\cong H^i(\mathcal{X},\Omega^{+,n}).
	\]
\end{proposition}
In order to prove this claim, we first need to study the sheaf $\Omega^{+,n}$ over the strongly étale site associated to a discrete local Huber pair $(A,A^+)$. 

\begin{lemma}
Let $(A,A^+)$ be a discretely ringed local Huber pair. Then $\Omega^{n,+}(Y)=\Omega^{n,\log}(Y)$ for any Cartesian strongly \'etale map $Y\to \Spa(A,A^+)$.
\end{lemma}
\begin{proof}
By \cite[Corollary 12.2]{HTame}, every strongly \'etale covering of $\Spa(A,A^+)$  has a refinement by \textit{Cartesian} morphisms: these are strongly \'etale coverings 
\[
\Spa(B,B^+)\to \Spa(A,A^+)
\]
for which 
\begin{equation}
	\begin{tikzcd}
		\Spec(B) \ar[r]\ar[d]&\Spec(B^+)\ar[d]\\
		\Spec(A)\ar[r]& \Spec(A^+)
	\end{tikzcd}
\end{equation}
is Cartesian. By \cite[Proposition 11.10]{HTame} the morphism $A^+\to B^+$ is \'etale. We conclude that $\Spa(A,A^+)_{\textrm{Cart,s\'et}}$ and $\Spec(A^+)_\et$ are equivalent, where $\Spa(A,A^+)_{\textrm{Cart,s\'et}}$ is the site of Cartesian strongly étale maps over $\Spa(A,A^+)$. 

We have inclusions of presheaves on $\Spa(A,A^+)_{\textrm{s\'et}}$
\(\Omega^{n,\log}\subset \Omega^{+,n}\subset \Omega^n.\)
But $\Omega^{+,n}$ is the sheafification of $\Omega^{n,\log}$ in $\Spa(A,A^+)_\sett$ and therefore in $\Spa(A,A^+)_{\textrm{Cart,\sett}}$. The presheaf
\[
\Spec(A^+)_\et\ni Y \mapsto \Omega^{n,\log}(\Spa(Y\otimes_{A^+}A,Y))
\]
is a quasi-coherent sheaf in the étale topology (see \cite[Section 03DR]{Stacks}). This follows from the following two fact
\begin{itemize}
	\item By \cite[Chapter IV, Corollary 1.2.8]{Ogus}, $\Omega^{n}_Y(\log Y\otimes_{A^+}A)$ is a sheaf in the Zariski topology of $Y$.
	\item The log structure on $Y$ is induced by the log structure on $A$. Thus the morphism of log schemes $(Y\otimes_{A^+}A,Y)\to (\Spec A,\Spec A^+)$ is strict, which implies that $\Omega^{n,\log}(\Spa(Y\otimes_{A^+}A,Y))$ the pullback of $\Omega_{(A,A^+)}^{n,\log}$.
\end{itemize}
 Hence  $\Omega^{n,\log}$ is already a sheaf on $\Spa(A,A^+)_{\textrm{Cart,s\'et}}$. The claim follows.
\end{proof}

\begin{proof}[Proof of \cref{StronglyEtaleEqualOpen}]
	To prove the claim we need to check that 
	\[
	H^i_{\textrm{s\'et}}(\mathcal{X},\Omega^{+,n})=0
	\]
	when $\mathcal{X}=\Spa(A,A^+)$ is a discretely ringed local adic space. We follow the steps of \cite[Proposition 12.3]{HTame}. As in loc. cit. we only need to prove that the \v{C}ech complex for $\Omega^{+,n}$ associated to \textit{Cartesian} strongly \'etale coverings 
	\[
	\Spa(B,B^+)\to \Spa(A,A^+)
	\]
	is trivial. Over such coverings $\Omega^{+,n}=\Omega^{n,\log}$.
	
	 Taking integral closures commutes with \'etale base change. Therefore $B^+\otimes_{A^+}\dots\otimes_{A^+}B^+$ is integrally closed in $B\otimes_A\dots\otimes_AB$.
 \[\resizebox{\columnwidth}{!}{$
 \Spa(B,B^+)\times_{\Spa(A,A^+)} \cdots \times_{\Spa(A,A^+)}\Spa(B,B^+)\cong \Spa(B\otimes_A\dots\otimes_AB, B^+\otimes_{A^+}\dots\otimes_{A^+}B^+).$}
 \]  
 	Furthermore, we may interpret $\Omega^{n,\llog}$ as a quasi-coherent sheaf over $\Spec(A^+)$   \cite[Chapter IV, Corollary 1.2.8]{Ogus}). Therefore the \v{C}ech complex in question is equal to the \v{C}ech complex of the quasi-coherent sheaf $\Omega^{n,\log}$, which is exact since $A^+\to B^+$ is faithfully flat.
\end{proof}
We furthermore compare the strongly \'etale and tame cohomology of $\Omega^{+,n}$.
\begin{proposition}\label{TameEqualStronglyEtale}
	Let $\mathcal{X}$ be a discretely ringed adic space over a valued field $\Spa(k,k^+)$. Then 
	\[
	H^i_{\textrm{s\'et}}(\mathcal{X},\Omega^{+,n})=H^i_t(\mathcal{X},\Omega^{+,n}).
	\]
\end{proposition}
\begin{proof}
	Under the assumption that $\mathcal{X}$ is defined over a field of characteristic $p$, the claim follows from \cite[Proposition 8.5]{HTame}, as $\Omega^{+,n}_\mathcal{X}$ is $p$-torsion. 
\end{proof}

Lastly, we aim to prove that  $H^{i}(\Spa(A,A^+),\Omega^{+,n})$ is trivial for log smooth Huber pairs $(A,A^+)$ with discrete topology. 
\begin{theorem}\label{CohomologyOfOmega+IsZero}
	Let $\mathcal{X}=\Spa(A,A^+)$ be a discretely ringed affinoid over a perfect field $k$ of positive characteristic $p$. Assume that $(A,A^+)$ is associated to a log smooth log scheme. Assume Conjectures \ref{ConjResolution3} and \ref{GeneralResolution}. Then
	\[
	H^i_t(\Spa(A,A^+),\Omega^{+,n})=H^i_{\textrm{s\'et}}(\Spa(A,A^+),\Omega^{+,n})=H^i(\Spa(A,A^+),\Omega^{+,n})=0,
	\] 
	for all $i\geq 1$.
\end{theorem}
\begin{proof}
	We proved in \cref{StronglyEtaleEqualOpen} and \cref{TameEqualStronglyEtale} that the tame, strongly \'etale and open cohomology groups of $\Omega^{+,n}$ agree.
	We have proved in \cref{CohomologyIsColimit} that the latter is isomorphic to
	 $\textrm{colim}_{\bar{Y}}H^i(\bar{Y},\Omega^{+,n}|_{\bar{Y}}),$
	where $\bar{Y}$ are $A$-modifications of $A^+$. As $A$ is of finite type over $A^+$, we may assume that these are compactifications. Under an adequate assumption of resolution of singularities (i.e. Conjectures \ref{ConjResolution3} and \ref{GeneralResolution}) we may assume that $\bar{Y}$ is log smooth over $k$ with the log structure associated to $\bar{Y}\backslash \Spec(A)$. Furthermore, as discussed in \cref{DescriptionOfOpens}, we may dominate these modifications by chains of blow ups at smooth centers, intersecting the boundary transversally.  Hence $\Omega^+|_{\bar{Y}}=\Omega^{n,\llog}_{\bar{Y}}$ by \cite[Theorem 8.12]{LogDiff}. Using \cref{AcyclicityProposition}, we obtain $H^i(\bar{Y},\Omega^+|_{\bar{Y}})=H^i(\bar{Y},\Omega^{n,\llog}_{\bar{Y}})=0$. This is proven by induction on the length of the chain of blow-ups:
	
	Let $\pi:\bar{Y}_2\to \bar{Y}_1$ be a morphism in the chain of blow ups. Denote its smooth center by $Z$. Note that $\Spec(A)$ is open in $\bar{Y}_i$ and that $\bar{Y}_i\backslash \Spec(A)$ is an snc divisor for $i=1,2$.  Assume that we know that $H^i(\bar{Y}_1, \Omega^{n,\log})=0$. Cover $\bar{Y}_1$ by affine opens $U_i$. Set $V_i:=\bar{Y}_2\times_{\bar{Y}_1}U_i$. By \cref{AcyclicityProposition}, we know that $H^i(V_i,\Omega^{n,\log}) =0$. We apply \cite[Proposition 02O5]{Stacks}, to conclude that $R^i\pi_*\Omega^{+,n}=0$ and hence
	\[
	H^i(\bar{Y}_2, \Omega^{+,n})=0. \qedhere
	\]
\end{proof}

%% file: PurityPrime.tex
\subsection{The Gysin morphism on \texorpdfstring{$\Omega^{+,n}$}{Omega⁺}}
\input{GysinMorphism}
\subsection{Purity for the tame topology in positive characteristic}
\input{PurityTame}
\subsection{Purity for regular pairs}
\input{PurityRegular}

%% file: GysinMorphism.tex
In this section we define the Gysin isomorphism for the sheaf of $\mathcal{O}^+$-modules $\Omega^{+,n}$. From now on the functor $i^!$ is taken in the tame topology, unless we explicitly say otherwise.
\begin{theorem}\label{Ri!IsConcentratedInDegree1}
	Let $k$ be a perfect field of characteristic $p$. Assume resolution of singularities and embedded resolution of singularities with boundary over $k$. Let $(X,Z)$ be a smooth pair over $k$ of pure codimension $1$. Set  $\mathcal{X}:=\Spa(X,k)$ and $\mathcal{Z}:=\Spa(Z,k)$. Then the sheaf $R^ni^!\Omega^{+,n}_{\mathcal{X}}$ is trivial for $n\neq 1$.
\end{theorem}
\begin{proof}
	Triviality of the sheaves $R^ni^!\Omega^{+,n}_{\mathcal{X}}$ for $n\geq 2$ can be checked on stalks. Let $x$ be a point in $\mathcal{Z}\subset\mathcal{X}$. Then when $n\geq 2$, the group $(R^ni^!\Omega^+_{\mathcal{X}})_x$ is isomorphic to the group
	\[
	\begin{split}
		\colim_{x\in\mathcal{U}\xrightarrow{\textrm{tame}} \mathcal{X}}H^n_{\mathcal{Z}\cap \mathcal{U}, t}(\mathcal{U},\Omega^{+,n})\cong &\colim_{x\in \mathcal{U}\xrightarrow{\textrm{tame}} \mathcal{X}}H^{n-1}_{t}(\mathcal{U}\backslash\mathcal{Z},\Omega^{+,n})\\ \cong&\colim_{x\in \mathcal{U}\xrightarrow{\textrm{tame}} \mathcal{X}}H^{n-1}(\mathcal{U}\backslash\mathcal{Z},\Omega^{+,n}) 
	\end{split}
\]
	where the colimit is taken over affinoids $\mathcal{U}$ containing $x$ and which are tame over $\mathcal{X}$.
	
	By assuming resolution of singularities (\cref{ConjResolution3}) as well as embedded resolution of singularities with boundary (\cref{GeneralResolution}), the category of tame neighborhoods $\{x\in \mathcal{U}\to \mathcal{X}\}$ is dominated by adic spaces of the form $\Spa(U,\bar{U})$, such that $(U\backslash Z,\bar{U})$ is a pair associated to a log smooth log scheme.
	
	
	 Since $Z$ has codimension $1$ in $X$, we may assume that both \( U \setminus Z \) and \( \bar{U} \) are affine, after replacing \( \bar{U} \) by an affine open subscheme. Then, by \cref{CohomologyOfOmega+IsZero}, we conclude that
	 $$H^{n-1}(\mathcal{U}\backslash \mathcal{Z},\Omega^{+,n})= H^{n-1}(\Spa(U \backslash Z,\bar{U}),\Omega^{+,n})=0 \textrm{ for } n\geq 2.$$
	 
	 It remains to check that \( i^! \Omega^{+,n}_{\mathcal{X}} = 0 \). By definition,
	 \[
	 i^! \Omega^{+,n}_{\mathcal{X}}(\mathcal{U}\cap \mathcal{Z})=\ker \big(\Omega^{+,n}_{\mathcal{X}}(\mathcal{U})\to  \Omega^{+,n}_{\mathcal{X}}(\mathcal{U}\backslash \mathcal{Z})\big)
	 \]
	 for every open $\mathcal{U}\subset \mathcal{X}$.
	
	By \cite[Lemma 8.8]{LogDiff}, $i^! \Omega^{n,\log}_{\mathcal{X},\lim}(\mathcal{U}\cap \mathcal{Z})$ is trivial. Since under our assumptions the equality \( \Omega^{n,\log}_{\mathcal{X}, \lim} = \Omega^{+,n} \) holds, we conclude that \( i^! \Omega^{+,n}_{\mathcal{X}}(\mathcal{U} \cap \mathcal{Z}) = 0 \).
\end{proof}

\begin{theorem}\label{ValueOfRi!OmegaCodim1}
	Let $k$ be a perfect field of characteristic $p$. Let $(X,Z)$ be a smooth pair over $k$ of pure codimension $1$.
	 Assume that Conjectures \ref{ConjResolution3} and \ref{GeneralResolution} hold.
	  Define $\mathcal{X}:=\Spa(X,k)$ and $\mathcal{Z}:=\Spa(Z,k)$. Then we have an isomorphism of $\mathcal{O}^+_\mathcal{Z}$-modules:
	\[
	R^1i^!\Omega_{\mathcal{X}}^{+,n}\to \Omega_{\mathcal{Z}}^{+,n-1} .
	\]
\end{theorem}
\begin{proof}
	Define $\mathcal{U}:=\mathcal{X}\backslash \mathcal{Z}$ and let $j:\mathcal{U}\to \mathcal{X}$ be the natural inclusion. By the excision long exact sequence we know that $R^1i^!\Omega_{\mathcal{X}}^{+,n}$ is isomorphic to the cokernel of $j_*\Omega_{\mathcal{U}}^{+,n}\to  \Omega_{\mathcal{X}}^{+,n}$. 
	
	By resolution of singularities, we may assume that \( \mathcal{U} = \Spa(U, \bar{U}) \) where \( \bar{U} \) is smooth, and both \( D := \bar{U} \setminus U \) and \( \bar{U} \setminus (U \setminus Z) = D \cup \bar{Z} \) are simple normal crossings divisors. Here, \( \bar{Z} \subseteq \bar{U} \) denotes the schematic image of \( Z \cap U \) in $\bar{U}$.
	
	 We may now use the second residue sequence in \cref{ResidueSequences}, stating that 
	\[
	0\to \Omega^{n,\log}_{\bar{U}}(\log (D+\bar{Z}))\to \Omega^{n,\log}_{\bar{U}}(\log (D))\to \Omega^{n-1,\log}_{\bar{Z}}(D|_{\bar{Z}})\to 0.
	\]
	
	The global sections of these terms are identified with the sections of \( \Omega^{+,n}_{\mathcal{X}}(\mathcal{X}) \), \( \Omega^{+,n}_{\mathcal{U}}(\mathcal{U}) \), and \( \Omega^{+,n-1}_{\mathcal{Z}}(\mathcal{Z}) \), respectively. Therefore, we obtain an exact sequence
	\begin{equation}\label{SESOmega+1}
		0 \to \Omega^{+,n}_{\mathcal{X}} \to j_* \Omega^{+,n}_{\mathcal{U}} \to i_* \Omega^{+,n-1}_{\mathcal{Z}} \to 0,
	\end{equation}
	which completes the proof.
\end{proof}
\begin{remark}
	\begin{itemize}
		 \item 	In the case $n=0$, the claim is equivalent to $R^1i^! \mathcal{O}^+_{\mathcal{X}}\cong 0$. This is true since $j_*\mathcal{O}^+_{\mathcal{U}}\cong \mathcal{O}^+_{\mathcal{X}}$, even without assuming resolution of singularities.
		 \item  In a future work, we give a proof of \cref{ValueOfRi!OmegaCodim1} without the assumption of Resolution of singularities.
	\end{itemize}

\end{remark}
\begin{corollary}\label{ValuesOfRi!Omega+}
	Let $k$ be a perfect field of positive characteristic $p$. Let $(X,Z)$ be a smooth pair of pure codimension $r$. Assume that resolution of singularities holds (Conjectures \ref{ConjResolution3} and \ref{GeneralResolution}). Define $\mathcal{X}:=\Spa(X,k)$ and $\mathcal{Z}:=\Spa(Z,k)$. Then we have an isomorphism in $D(\mathcal{Z})$:
	\[
	\phi_{\mathcal{Z}/\mathcal{X}}:\Omega_{\mathcal{Z}}^{+,n-r}[-r]\to Ri^!\Omega_{\mathcal{X}}^{+,n}
	\]
\end{corollary}
\renewcommand\qedsymbol{$\underset{\cref{ValuesOfRi!Omega+} \textrm{ when } Z \textrm{ is a complete intersection}}{\square}$}
\begin{proof}
	 This is done by induction on the codimension of $Z$. When $\textrm{codim}_Z(X)=1$ this is proven by combining Theorems \ref{CohomologyOfOmega+IsZero} and \ref{ValueOfRi!OmegaCodim1}. For general codimension $r$, assume that $Z$ is a complete intersection of smooth schemes. Then there exists a smooth subscheme $Z'$, which is also a complete intersection in $X$, such that
	 $Z\xrightarrow{i_1} Z' \xrightarrow{i_2} X $,  and $\textrm{codim}_{Z'}(X)=r-1$. 
	   We obtain
	 \[
	 	Ri^!\Omega^{+,n}_{\mathcal{X}}\cong R(i_2 \circ i_1)^!\Omega^{+,n}_{\mathcal{X}}\cong Ri_1^! \big(Ri_2^! \Omega^{+,n}_{\mathcal{X}}\big)\cong Ri_1^! \Omega^{+,n-r+1}_{\Spa(Z',k)}[-r+1]\cong \Omega^{+,n-r}_{\mathcal{Z}}[-r]. \qedhere
	 \]
\end{proof}
\renewcommand\qedsymbol{$\square$}
Milne's proof of semi-purity for log de Rham-Witt sheaves \cite[Proposition 2.1]{MilneValuesOfZeta} uses the following proposition
\begin{proposition}
	Let $k$ be a field. 
	There is a family of homomorphisms $$\phi_{Z/X}:\Omega^{s-r}_{Z/k}\to R^{r}i^!\Omega_{X/k}^s,$$
	defined for every smooth pair $(X,Z)$ of codimension $r$, where $i:Z\to X$ is the inclusion. Furthermore, $\Phi$ is defined by  with the following properties
	\begin{itemize}
		\item $\phi_{Z/X}$ commutes with the differentials $d$.
		\item If $Z\xrightarrow{i_{Z/Y}} Y\xrightarrow{i_{Y/X}} X$ are closed immersions of smooth schemes over $k$, such that $(X,Z)$ is of codimension $r$ and $(Y,Z)$ is of codimension $t$, then $$R^{r-t}i^!_{Z/Y}(\phi_{Y/X})\circ \phi_{Z/Y}=\phi_{Z/X}.$$
		\item Assume that $r=1$. Let $z\in Z$ be a closed point and let $A=\Spec(\mathcal{O}_{X,z})$. Let $f=0$ be a local equation of $Z$ around $z$. Then $(R^1i^!\Omega^s_{X/k})_z$ is associated to the $A$-module $\Omega^s_{A[f^{-1}]}=\Omega_A^s$ and $(\phi_{Z/X})_z$ is associated to the residue map
		\[
		\omega\mapsto \omega \wedge \dlog(f): \; \Omega_{A/(f)}^{s-1}\to \Omega_{A[f^{-1}]}^s/\Omega_A^s.
		\]
	\end{itemize} 
\end{proposition}
However Milne gives the reference \cite[III.8]{HartshorneAmple}, which only states the proposition in characteristic $0$. The construction does however work in any characteristic (see the first part of the proof of \cite[Theorem III.8.1]{HartshorneAmple}). We obtain a morphism of sheaves over adic spaces
\[
\phi_{\mathcal{Z}/\mathcal{X}}:\Omega_{\mathcal{Z}/k}^{s-r}\to R^ri^!\Omega_{\mathcal{X}/k}^s,
\]
where $\mathcal{Z}=\Spa(Z,k)$ and $\mathcal{X}=\Spa(X,k)$.

\begin{proposition}\label{CommutingLocal}
	Under the assumptions of \cref{ValuesOfRi!Omega+}, the diagram
	\begin{equation}\label{DiagramOfPhis}
		\begin{tikzcd}
			\Omega_{\mathcal{Z}/k}^{+,s-r}\ar[r]\ar[d,swap,"\phi_{\mathcal{Z}/\mathcal{X}}"]&\Omega_{\mathcal{Z}/k}^{s-r}\ar[d,"\phi_{\mathcal{Z}/\mathcal{X}}"]\\ 
			R^ri^!\Omega_{\mathcal{X}/k}^{+,s}\ar[r]& R^ri^!\Omega_{\mathcal{X}/k}^s.
		\end{tikzcd}
	\end{equation}
	commutes.
\end{proposition}
\begin{proof}
	We start by assuming $r=1$. Consider a point $z\in \mathcal{Z}$. Let $(A,A^+)$ the local Huber pair associated to $\mathcal{X}_z$. There exists $f\in A$ such that $\mathcal{Z}_z=\Spa(A/(f),A^+/(f))$. The diagram \eqref{DiagramOfPhis} corresponds locally to the diagram
	\begin{equation}\label{DiagramOfPhisLocal}
		\begin{tikzcd}
			\Omega_{(A/(f),A^+/(f))}^{+,s-1}\ar[r]\ar[d,swap,"\omega\mapsto \omega\wedge\dlog f"]&\Omega_{A/(f)}^{s-1}\ar[d,"\omega\mapsto \omega\wedge\dlog f"]\\
			\Omega^{+,s}(\Spa(A[f^{-1}],A^+))/\Omega^{+,s}_{(A,A^+)}\ar[r]&\Omega^{s}_{A[f^{-1}]}/\Omega^{s}_{A}.
		\end{tikzcd}
	\end{equation}
	Since \eqref{DiagramOfPhisLocal} commutes, the diagram \eqref{DiagramOfPhis} also commutes for $r=1$. The claim for $r>1$ follows by induction.
\end{proof}
\begin{proof}[End of the proof of \cref{ValuesOfRi!Omega+}]
	In the proof of \cite[Théorème 3.5.8]{CohomologyLogGros}, Gros proves that the morphism
	\[
	\phi_{\mathcal{Z}/\mathcal{X}}:\Omega_{\mathcal{Z}/k}^{s-r}\to R^ri^!\Omega_{\mathcal{X}/k}^s,
	\]
	is an injection. Since $\Omega^{+,s-r}_{\mathcal{Z}/k}\to \Omega^{s-r}_{\mathcal{Z}/k}$ is also injective, the morphism 
	\[
	\Omega^{+,s-r}_{\mathcal{Z}/k}\to\Omega_{\mathcal{Z}/k}^{s-r}\xrightarrow{\Phi_{\mathcal{Z}/\mathcal{X}}}R^ri^!\Omega_{\mathcal{X}/k}^s
	\]
	is also injective and factors uniquely over $R^ri^!\Omega^{+,s}_{\mathcal{X}/k}$ locally (due to \eqref{DiagramOfPhisLocal}). We conclude that 
	\[
	\phi_{\mathcal{Z}/\mathcal{X}}: \Omega^{+,s-r}_{\mathcal{Z}/k}\to R^ri^!\Omega^{+,s}_{\mathcal{X}/k}
	\]
	is defined globally.
\end{proof}
\begin{remark}
	Note that, since the fact that diagram \eqref{DiagramOfPhis} commutes can be checked locally, \cref{CommutingLocal} holds even when $Z$ is not a complete intersection.
\end{remark}
\begin{corollary}\label{ValueOfRi!ZOmega}
	Assume the Conjectures \ref{ConjResolution3} and \ref{GeneralResolution}. 
	
	Let $i:Z\to X$ be a smooth closed immersion of pure codimension $r$. Let $i:\mathcal{Z}\to \mathcal{X}$ the associated inclusion of discretely ringed adic spaces. Then we have an isomorphism
	\[
	\phi_{\mathcal{Z}/\mathcal{X}}:Z\Omega^{n-r,+}_{\mathcal{Z}}[-r]\to R i^! Z\Omega^{n,+}  .
	\]
\end{corollary}
\begin{proof}
	We first prove the claim when \underline{$r=1$}. \\
	\textbf{Step 1: Proof that $Ri^! Z\Omega^{+,n}$ and $Ri^!B\Omega^{+,n}$ are concentrated in degree $1$:} 
	
	This is done by induction on  the degree $n$. For degree $n=0$, $Z\Omega^{+,0}=(\mathcal{O}^{+})^p$ consists of sections of $\mathcal{O}^+$ which are $p$-powers. This is isomorphic to $\mathcal{O}^+$ by the Frobenius morphism. The claim follows for $Z\Omega^{+,0}$. Since $B\Omega^{+,0}=0$, the case $n=0$ is complete.
	
	Note that since $Z\Omega^{+,n}\cong Z\Omega^{n,\log}_{\lim}$ and $B\Omega^{+,n}=B\Omega^{n,\log}_{\lim}$, we have $i^!Z\Omega^{+,n}=i^!B\Omega^{+,n}=0$. This follows for any sheaf $\mathcal{F}_{\lim}$ when $\mathcal{F}$ is unramified. Indeed \cite[Lemma 8.8]{LogDiff} states that for $\mathcal{U}\subset \mathcal{U}'$, one has an inclusion
	\[
	\mathcal{F}_{\lim}(\mathcal{U}')\subset \mathcal{F}_{\lim}(\mathcal{U}).
	\]
	For a given tame morphism $\mathcal{U}\to \mathcal{X}$ the group of section $i^!\mathcal{F}_{\lim}(\mathcal{U})$ is the kernel of $\mathcal{F}_{\lim}(\mathcal{U}\backslash\mathcal{Z})\to \mathcal{F}_{\lim}(\mathcal{U})$, which is injective. Hence $i^!\mathcal{F}_{\lim}(\mathcal{U})=0$.

	 It is left to check by induction that $R^{j}i^!Z\Omega^{+,n}=R^ji^!B\Omega^{+,n}=0$ when $j\geq 2$.
	
	Assume the claim holds for $n\in \mathbb{N}$. Consider the short exact sequence in \cref{ShortExactOnC}
	\[
	0\to B\Omega^{+,n+1}\to Z\Omega^{+,n+1}\xrightarrow{C} \Omega^{+,n+1}\to 0,
	\]
	as well as the short exact sequence 
	\[
	0\to Z\Omega^{+,n}\to \Omega^{+,n}\to B\Omega^{+,n+1}\to 0.
	\]
	If we consider the long exact sequence associated to the latter we obtain  that $R^ji^! B\Omega^{+,n+1}$ is trivial when $j\geq 2$. Applying the long exact sequence of cohomology to the former implies that $R^ji^!Z\Omega^{+,n+1}$ is trivial when $j\geq 2$.\\	
	\textbf{Step 2: Computing $Ri^!Z\Omega^{+,n}$} 
	
	Since $Ri^!\Omega^{+,n+1}_{\mathcal{X}}$ is concentrated in degree one when $r=1$ we may apply the residue sequences  in \cref{ResidueSequenceClosedForms} together with Conjectures \ref{ConjResolution3} and \ref{GeneralResolution} to conclude that $Ri^!Z\Omega^{+,n+1}_{\mathcal{X}}=Z\Omega^{+,n+1}_{\mathcal{Z}}$ by a similar argument to the one used in \cref{ValueOfRi!OmegaCodim1}.
	
	When \underline{$r\geq 1$}, we find $Z\subset Z'\subset X$ locally,  with corresponding immersions $Z\xrightarrow{i_1} Z' \xrightarrow{i_2} X$ and $\textrm{codim}_{Z'}(X)=r-1$. We obtain
	\[
	\begin{split}
	 \phi_{\mZ/\mX}:	Ri^!Z\Omega^{+,n}_{\mathcal{X}}\cong & R(i_2 \circ i_1)^!Z\Omega^{+,n}_{\mathcal{X}}\cong  Ri_1^! \big(Ri_2^! \Omega_{\mathcal{X}}\big) \\ \cong & Ri_1^! Z\Omega^{+,n-r+1}_{\Spa(Z',k)}[-r+1]\cong Z\Omega^{+,n-r}_{\mathcal{Z}}[-r]. 
	\end{split}
	\]

Since $\phi_{\mathcal{Z}/\mathcal{X}}$ commutes with differentials in both instances, one also gets a diagram for the closed differentials
	\begin{equation}\label{DiagramOfPhisZ}
	\begin{tikzcd}
		Z\Omega_{\mathcal{Z}/k}^{+,s-r}\ar[r]\ar[d,swap,"\phi_{\mathcal{Z}/\mathcal{X}}"]&Z\Omega_{\mathcal{Z}/k}^{s-r}\ar[d,"\phi_{\mathcal{Z}/\mathcal{X}}"]\\ 
		R^ri^!Z\Omega_{\mathcal{X}/k}^{+,s}\ar[r]& R^ri^!Z\Omega_{\mathcal{X}/k}^s.
	\end{tikzcd}
\end{equation}
Therefore $\phi_{\mZ/\mX}$ is independent from the choice of $Z'$, hence may be defined globally.
\end{proof}

%% file: PurityTame.tex
\'Etale cohomology fails to satisfy purity for smooth pairs. Consider for example $i:\{0\}\to \mathbb{A}^1$. We have $\nu_{\mathbb{A}^1}(0)=\mathbb{Z}/p\mathbb{Z}$. Hence there is an exact sequence
\[
0 \to R^1 i^{!} \mathbb{Z}/p\mathbb{Z}\to R^1i^{!} \mathbb{G}^p_{a,\mathbb{A}^1}\xrightarrow{F^{-1}-1} R^1i^{!} \mathbb{G}_{a,\mathbb{A}^1}\to R^2i^{!} \mathbb{Z}/p\mathbb{Z}\to 0.
\]

Note that $R^1i^{!} \mathbb{G}_{a,\mathbb{A}^1}$ is locally isomorphic to $H^1_Z(X,\mathbb{G}_a)$, which is isomorphic to $\textrm{coker}(\mathcal{O}_{X,x}\to \mathcal{O}_{X,x}[t^{-1}])$ because of the excision sequence and the fact that \'etale cohomology of quasi-coherent sheaves is zero on affines.

But $F^{-1}-1$ is not surjective on $\mathcal{O}_{X,x}[t^{-1}]$. Indeed, for $x=0\in \mathbb{A}^1$, finding a preimage of $1/t$ corresponds to solving the equation $T^p-T=1/t$, or equivalently finding roots of the polynomial $f(T)=tT^p-tT-1$. But $f'(T)=-t$ which lies in $m_x$, thus $k[T]/(tT^p-tT-1)\to k[T]$ is not \'etale. In conclusion the preimage of $1/t$ under $F^{-1}-1$ does not exist in the \'etale topology.

The fact that $Z\Omega_{X,x}^{n}[t^{-1}]\xrightarrow{C-1} \Omega_{X,x}^n[t^{-1}]$ is not surjective seems to be an obstruction to purity. It is worth noting that the Cartier operator and $F^{-1}$ coincide when $n=0$. 

To remedy this, an idea is to replace the sheaf $\Omega$ with $\Omega^+$ and \'etale topology with tame topology. 

\begin{theorem}[c.f. \cref{MainTheorem}]\label{MainTheoremAtTheEnd}
Let $k$ be a perfect field of positive characteristic $p$. Let $(X,Z)$ be a smooth pair over $k$ of pure codimension $r$. Assume that Conjectures \ref{ConjResolution3} and \ref{GeneralResolution} hold. Define $\mathcal{X}:=\Spa(X,k)$ and $\mathcal{Z}:=\Spa(Z,k)$. Then we have an isomorphism in $D(\mathcal{Z}_t)$:
\[
Gys_{m}:\nu_{\mathcal{Z},m}(n-r)[-r]\to Ri^!\nu_{\mathcal{X},m}(n), \quad \forall m,n \in \mathbb{N}.
\]
\end{theorem}

\renewcommand\qedsymbol{$\overset{\ref{MainTheoremAtTheEnd} \textrm{ for } m=1}{\square}$}
\begin{proof}
	We start by assuming $m=1$. The short exact sequence
	\[
	0\to \nu_{\mathcal{X}}(n)\to Z\Omega^{+,n}_{\mathcal{X}}\xrightarrow{C-1}\Omega^{+,n}_\mathcal{X}\to 0
	\]
	induces an exact sequence 
	\[
	0\to R^ri^!\nu_{\mathcal{X}}(n)\to R^ri^!Z\Omega^{+,n}_{\mathcal{X}}\xrightarrow{C-1}R^ri^!\Omega^{+,n}_\mathcal{X}\to R^{r+1}i^!\nu_{\mathcal{X}}(n)\to 0
	\]
	since $Ri^!\Omega^{+,n}_\mathcal{X}$ and $Ri^!Z\Omega^{+,n}_\mathcal{X}$ are concentrated in degree $r$ as established in the previous section. It is enough to prove that the square on the right hand side of the diagram
	\begin{equation}\label{FinalDiagram}
		\begin{tikzcd}
			0\ar[r] &\nu_{\mathcal{Z}}(n-r)\ar[r]\arrow[dotted,d]&Z\Omega_{\mathcal{Z}}^{+,n-r}\ar[r,"C-1"]\ar[d,"\cong \;(\textrm{Corollary } \ref{ValueOfRi!ZOmega})"]&\Omega^{n-r}_{\mathcal{Z}}\ar[r]\arrow[d,"\cong \;(\textrm{Corollary } \ref{ValuesOfRi!Omega+})"] &0\\
				0\ar[r] &R^ri^!\nu_{\mathcal{X}}(n)\ar[r]&R^ri^!Z\Omega_{\mathcal{X}}^{+,n}\ar[r,"C-1"]&R^ri^!\Omega_{\mathcal{X}}^{+,n}\ar[r]&R^{r+1}i^!\nu_{\mathcal{X}}(n),
		\end{tikzcd}
	\end{equation}
	commutes. By induction on $r$, it suffices to prove this for $r=1$.
	
	As in the proof of \cref{ValueOfRi!OmegaCodim1} and \cref{ValueOfRi!ZOmega} it suffices to consider the commutativity at the  level of sections for an open $\mathcal{U}=\Spa(U,\bar{U})\subset \mathcal{X}$ satisfying that $\bar{U}$ is smooth over $k$ and $\bar{U}\backslash U$ as well as $Z\cup (\bar{U}\backslash U)=:\bar{Z}\cup D$ are snc divisors, where $\bar{Z}$ is the schematic image of $Z\cap U$ in $\bar{U}$. 
	
	The isomorphisms
	$R^1i^{!}\Omega^{+,n+1}_{\mathcal{X}}\to \Omega^{+,n} _{\mathcal{Z}}$ and $R^1i^!Z\Omega^{+,n+1}_{\mathcal{X}}\to Z\Omega^{+,n}_{\mathcal{Z}}$ correspond to the residue morphisms 
	\[
		\Omega^{n+1}_{\bar{U}}(\log (D+\bar{Z}))\to \Omega^{n}_{\bar{Z}}(\log D|_{\bar{Z}}) \; \textrm{ and } 	Z\Omega^{n+1}_{\bar{U}}(\log (D+\bar{Z}))\to Z\Omega^{n}_{\bar{Z}}(\log D|_{\bar{Z}}).
	\]
	It is therefore enough to check that the Cartier operator commutes with the restriction morphism.
	
	Assume that $\bar{Z}$ is the zero set associated to a section $f$ of $\mathcal{O}_{\bar{U}}$. The Cartier operator $C:Z\Omega_{\bar{U}}^{n+1}(\log(D+\bar{Z}))\to \Omega_{\Bar{U}}^{n+1}(\log(D+\bar{Z}))$ sends 
	\begin{equation}\label{FormOfEta}
		\eta=\gamma+\eta'\wedge d\log(f); \quad \gamma \in Z\Omega_{\bar{U}}^{n+1}(\log D) , \eta'\in \Omega_{\Bar{U}}^{n}(\log(D+\bar{Z}))
	\end{equation}
	 to $C(\eta)=C(\gamma)+C(\eta')\wedge d\log(f)$. By \cref{ResidueSequenceClosedForms} every $\eta\in Z\Omega_{\bar{U}}^{n+1}(\log(D+\bar{Z}))$ takes the form \eqref{FormOfEta}.

	  Since $\gamma$ is a section of $Z\Omega_{\bar{U}}^{n+1}(\log D) \subset Z\Omega_{\Bar{U}}^{n+1}(\log (D+\bar{Z}))$, its image $C(\gamma)$ is a section of $ \Omega_{\bar{U}}^{n+1}(\log D)$. Hence the image of $\eta$ and $C(\eta)$ by the residue map is $\eta'|_{\bar{Z}}$ and $C(\eta')|_{\bar{Z}}$ respectively.
	
	It is only left to check that the Cartier operator commutes with the restriction morphism to $\bar{Z}
$, i.e. $\eta'|_{\bar{Z}}=C(\eta')|_{\bar{Z}}$.	The inverse Cartier isomorphism
\[
\bar{C}^{-1}:\Omega_{\bar{U}}(\log (D+\bar{Z}))\to Z\Omega_{\bar{U}}(\log(D+\Bar{Z}_U))/B\Omega_{\Bar{U}}(\log(D+\bar{Z}))\]
sends elements of the form $hdg$ to $h^pg^{p-1}dg$, which is compatible with restriction to $\bar{Z}$.
Hence the Cartier operator $C$, which is the composition
\[
	Z\Omega_{\bar{U}}(\log(D+\bar{Z}))\to Z\Omega_{\bar{U}}(\log(D+\bar{Z}))/B\Omega_{\Bar{U}}(\log(D+\bar{Z})) \xrightarrow{\bar{C}}\Omega_{\bar{U}}(\log (D+\bar{Z})),
\]
commutes with the restriction to $\bar{Z}$. As wedge products commute with restriction and the Cartier operator, the claim follows. 
\end{proof}
\renewcommand\qedsymbol{$\square$}
The semi-purity isomorphism \cite[Proposition 2.1]{MilneValuesOfZeta}, induces by comparison between the \'etale cohomology of a scheme and of an adic space \cite[Lemma 8.3]{HTame} an isomorphism
\begin{equation}\label{PurityEtale}
	 \phi_{\mathcal{Z}/\mathcal{X}}:\nu_{\mathcal{Z},m}(n-r)\to R^ri^!_\et\nu_{\mathcal{X},m}(n).
\end{equation}
Let $j:\mathcal{U}:=\mathcal{X}\backslash \mathcal{Z}\to \mathcal{X}$ be the natural inclusion. When $r=1$, excision induces isomorphisms
\[
R^1i^!_\et\nu_{\mathcal{X},m}(n)\cong \big(j_*\nu_{\mathcal{U},m}(n)/\nu_{\mathcal{X},m}(n)\big)_\et \;\textrm{ and } \;
 R^1i^!_{t}\nu_{\mathcal{X},m}(n)\cong \big(j_*\nu_{\mathcal{U},m}(n)/\nu_{\mathcal{X},m}(n)\big)_{t},
\]
where the subscript $\et$ and $t$ indicates the topology in which $i^!$ and the quotient are taken. Consequently, $R^1i^!_\et\nu_{\mathcal{X},m}(n)$ is the \'etale sheafification of the tame sheaf $R^1i^!_{t}\nu_{\mathcal{X},m}(n)$. In order to make such a comparison, we work on the site $(\textrm{Ét}/\mathcal{X})_t$ instead of $\mathcal{X}_t$, whose objects are étale morphisms over $\mathcal{X}$, and coverings are tame coverings. This is possible due to \cite[Proposition 3.1]{MilneEtale}.
\begin{proposition}\label{CompatibilityGysin}
	Assume that we are in the situation of \cref{MainTheoremAtTheEnd}. Assume that $r=1$. The diagram 
	\[
	\begin{tikzcd}
		\nu_{\mathcal{Z},1}(n-1)\ar[swap,dr,"\Phi_{\mathcal{Z}/\mathcal{X}}"] \arrow[r,"Gys_1"]& R^1i^!_t\nu_{\mathcal{X},1}(n) \arrow[d,"\textrm{sheafification}"]\\
		& R^1i^!_\et \nu_{\mathcal{X},1}(n)
	\end{tikzcd}
	\]
	commutes. In particular, the sheafification morphism
	\[
	R^1i^!_{t}\nu_{\mathcal{X},1}(n) \to R^1i^!_\et \nu_{\mathcal{X},1}(n)
	\]
	is an isomorphism.
\end{proposition}
\begin{proof}
	Consider the diagram
	\[
	\begin{tikzcd}[row sep=2.5em]
		Z\Omega^{+,n-1}_{\mathcal{Z}}\arrow[rr] \arrow[dr,swap,"C-1"] \arrow[dd,swap,"\phi_{\mathcal{Z}/\mathcal{X}}"] &&
		Z\Omega^{n-1}_{\mathcal{Z}}\arrow[dd,swap,"\phi_{\mathcal{Z}/\mathcal{X}}" near start] \arrow[dr,"C-1"] \\
		&\Omega^{+,n-1}_{\mathcal{Z}}  \arrow[rr,crossing over] &&\Omega^{n-1}_{\mathcal{Z}}
		\arrow[dd,"\phi_{\mathcal{Z}/\mathcal{X}}"] \\
		R^1i^!Z\Omega^{+,n}_{\mathcal{X}} \arrow[rr] \arrow[dr,swap,"C-1"] && R^1i^!Z\Omega^{n}_{\mathcal{X}} \arrow[dr,swap,"C-1"] \\
		& R^1i^!\Omega^{+,n}_{\mathcal{X}}  \arrow[rr] \arrow[uu,<-,crossing over,"\phi_{\mathcal{Z}/\mathcal{X}}" near end]&& R^1i^!\Omega^{n}_{\mathcal{X}}.
	\end{tikzcd}
	\]
	Taking the kernels of all the morphisms denoted by $C-1$ delivers the commutative diagram
	\[
	\begin{tikzcd}
		\nu_{\mathcal{Z}}(n-r)\ar[r,"="]\ar[d,"Gys_1"]&\nu_{\mathcal{Z}}(n-r)\ar[d,"\phi_{\mathcal{Z}/\mathcal{X}}"]\\
		R^1i^!_{t}\nu_{\mathcal{X}}(n)\ar[r]&R^1i^!_{\et}\nu_{\mathcal{X}}(n),
	\end{tikzcd}
	\]
	since $i^!\Omega_{\mathcal{X}}^{+,n}$ and $i^!\Omega_{\mathcal{X}}^n$ are trivial.
\end{proof}
\renewcommand\qedsymbol{$\overset{\ref{MainTheoremAtTheEnd}}{\square}$}
\begin{proof}[Proof of \cref{MainTheoremAtTheEnd} for all $m$]
By applying $Ri^!$ on the short exact sequence 
	\[
0\to \nu_{\mathcal{X},m}(n)\to \nu_{\mathcal{X},m+1}(n)\to \nu_{\mathcal{X},1}(n)\to 0.
\]
defined in \cref{highermSES}, we conclude that $Ri^!\nu_{\mX,m}(n)$ is concentrated in degree $r$ for all $m,n\in \mN$. It is then left to find for each $m$ a morphism 
\[
Gys_m:\nu_{\mathcal{Z},m}(n-r)\to R^ri^!\nu_{\mathcal{X},m}(n),
\]
such that 
\[
\begin{tikzcd}
	0\ar[r]& \nu_{\mathcal{Z},m}(n-r)\ar[r]\ar[d,"Gys_m"]&\nu_{\mathcal{Z},m+1}(n-r)\ar[r]\ar[d,"Gys_{m+1}"] &\nu_{\mathcal{Z},1}(n-r)\ar[r]\ar[d,"Gys_1"] &0\\
	0\ar[r]&R^{r}i^!\nu_{\mathcal{X},m}(n)\ar[r]&R^{r}i^!\nu_{\mathcal{X},m+1}(n)\ar[r]&R^{r}i^!\nu_{\mathcal{X},1}(n)\ar[r]& 0,
\end{tikzcd}
\]
commutes.
We first assume that \underline{$r=1$}. By induction on $m$, there exists $Gys_m$ fitting in the following diagram
\[
\begin{tikzcd}
	0\ar[r]& \nu_{\mathcal{Z},m}(n-1)\ar[r]\ar[d,"Gys_m"]&\nu_{\mathcal{Z},m+1}(n-1)\ar[r]\ar[d,dotted,"Gys_{m+1}"] &\nu_{\mathcal{Z},1}(n-1)\ar[r]\ar[d,"Gys_1"] &0\\
	0\ar[r]&R^{1}i^!_{t}\nu_{\mathcal{X},m}(n)\ar[r]\ar[d,"\cong"]&R^{1}i^!_{t}\nu_{\mathcal{X},m+1}(n)\ar[r]\ar[d,"\cong"]&R^{1}i^!_{t}\nu_{\mathcal{X},1}(n)\ar[r]\ar[d,"\cong (\ref{CompatibilityGysin})"]& 0\\
	0\ar[r]&R^{1}i^!_{\et}\nu_{\mathcal{X},m}(n)\ar[r]&R^{1}i^!_{\et}\nu_{\mathcal{X},m+1}(n)\ar[r]&R^{1}i^!_{\et}\nu_{\mathcal{X},1}(n)\ar[r]& 0,\\	
\end{tikzcd}
\]
where the composition of each of the vertical sequences is $\phi_{\mathcal{Z}/\mathcal{X}}$. Notice that the lower short exact sequence is exact since it is isomorphic to the upper one. Furthermore, the lower vertical morphisms are the sheafification morphisms. These are bijective. Hence $Gys_{m+1}$ can be defined as the composition 
\[
\nu_{\mathcal{Z},m+1}(n-r) \xrightarrow{\phi_{\mathcal{Z}/\mathcal{X}}} R^{1}i^!_{\et}\nu_{\mathcal{X},m+1}(n) \xrightarrow{\textrm{sheafification}^{-1}} R^{1}i^!_{t}\nu_{\mathcal{X},m+1}(n).
\]
\textit{Proof for all $r$ and $m$:} We now assume that $Z$ has codimension $r>1$. Then, locally in the open topology, there exists a smooth closed subvariety $Y\subset X$ of codimension $r-1$ containing $Z$. Define $\mY:= \Spa(Y,k)$. The closed immersion $i:\mZ\to \mX$ can be factored as follows $$\mathcal{Z}\xrightarrow{i_1}\mathcal{Y}\xrightarrow{i_2}\mathcal{X}.$$
 Define $Gys_m$ on $\mathcal{X}$ as the composition of isomorphisms
\[
\nu_{\mathcal{Z},m}(n-r)[-r]\xrightarrow{Gys_m}Ri_{1,t}^!\nu_{\mathcal{Y},m}(n-r+1)[-r+1]\xrightarrow{Ri^!_{1,t}Gys_m} Ri^!_t\nu_{\mathcal{X},m}(n).
\]
In order to construct $Gys_m$ globally, we claim that $Gys_m$ is independent from the choice of $Y$. It is sufficient to prove that the diagram
\begin{equation}\label{HigherCompatibilityGysin}
\begin{tikzcd}
	\nu_{\mathcal{Z},m}(n)\ar[r,"\phi_{\mathcal{Z}/\mathcal{X}}"]\ar[d,"Gys_m"]&R^ri_\et^!\nu_{\mathcal{X},m}(n)\ar[d]\\
	R^ri_t^!\nu_{\mathcal{X},m}(n) \ar[r]& R^ri^!W_m\Omega_{\mathcal{X}}^n
\end{tikzcd}
\end{equation}   
commutes. Indeed, if this is the case, then $R^ri^!_t\nu_{\mathcal{X},m}(n)\to R^ri^!W_m\Omega_{\mathcal{X}}^n$ is an inclusion. As $\nu_{\mathcal{Z},m}(n)\xrightarrow{\Phi_{\mathcal{Z}/\mathcal{X}}} R^ri^!W_m\Omega_{\mathcal{X}}$ is also an inclusion, the morphism $\nu_{\mathcal{Z},m}(n)\xrightarrow{Gys_m}R^ri^!_t\nu_{\mathcal{X},m}(n)$ is unique and independent from the choice of $Y$.
		
 When $r=1$, this is true by the previous discussion. Assume that the claim is true for some $r$. The diagram \eqref{HigherCompatibilityGysin} for $r+1$ may be written as
\begin{equation*}
	\begin{tikzcd}
		\nu_{\mathcal{Z},m}(n-r-1)\ar[d]\ar[r]\ar[d,phantom,shift left=6ex,"\lcirc{{1}}"]&R^1i^!_{1,\et}\nu_{\mathcal{Y},m}(n-r)\ar[d,phantom,"\lcirc{2}"]\ar[r]&R^{r+1}i_\et^!\nu_{\mathcal{X},m}\ar[dd]\ar[dd,phantom,shift right=8ex,"\lcirc{{4}}"]\\
		R^1i_{1,t}^!\nu_{\mathcal{Y},m}(n-r)\ar[d,phantom,shift left=15ex,"\lcirc{{3}}"]\ar[r]\ar[ru]\ar[d]&R^1i_{1,t}^!R^ri_{2,\et}^!\nu_{\mathcal{X},m}(n)\ar[ru]\ar[rd]&\\
		R^{r+1}i^!_{t}\nu_{\mathcal{X},m}(n)\ar[rr]&&R^{r+1}i^!W_m\Omega^n_{\mathcal{X}},
	\end{tikzcd}
\end{equation*}
where $i$ is factored as $\mathcal{Z}\xrightarrow{i_1}\mathcal{Y}\xrightarrow{i_2}\mathcal{X}$ with $Y$ being of codimension $r$. The upper diagonal morphisms are \'etale sheafification morphisms. Each of the diagrams \lcirc{1} to \lcirc{4} commute:

\begin{itemize}
	\item	The diagram \lcirc{1} commutes by the codimension $1$ case. Notice here that $\mathcal{Z}$ has codimension $1$ in $\mathcal{Y}$.
	
	\item	The diagram \lcirc{{2}} commutes since sheafification is compatible with isomorphisms.
	\item The diagram and \lcirc{{3}} commutes by the induction step on $r$ and applying the functor $R^1i_{1,t}^!$.
	\item 	Lastly, the diagram \lcirc{{4}} commutes by the universal property of sheafification since the upper diagonal map is an \'etale sheafification map.
\end{itemize}    
As $Gys_m$ is independent from the choice of $\mathcal{Y}$, it glues to a global isomorphism.
\end{proof}
\begin{remark}
	In his proof of \cite[Proposition 2.1]{MilneValuesOfZeta}, Milne does not construct a Gysin morphism for $\nu_{m}(n)$ when $m>1$. The reference \cite{CohomologyLogGros} is more precise in that aspect.
\end{remark}
\renewcommand\qedsymbol{$\square$}
\begin{corollary}\label{MainResultCorollary}
	Under the same assumptions as in \cref{MainTheoremAtTheEnd}, we have
	\[
		H^l_{\mathcal{Z},t}(\mathcal{X},\nu_m(n))=H^{l-r}_t(\mathcal{Z},\nu_m(n-r)).
 \]
\end{corollary}
\begin{proof}
	We consider the spectral sequence
	\[
	R^l \Gamma (\mathcal{Z}, R^si^! \nu_m(n))\implies R^{s+l}\Gamma_{\mathcal{Z}}(\mathcal{X},\nu_m(n)).
	\]
	Since $Ri^! \nu_m(n)$ is concentrated in degree $r$ and $R^ri^!\nu_m(n)=\nu_{m}(n-r)[-r]$, we have 
	\[
	R^l \Gamma (\mathcal{Z},\nu_m(n-r)[-r]) \cong  R^{r+l}\Gamma_{\mathcal{Z}}(\mathcal{X},\nu_{m}(n-r))\qedhere
	\]
\end{proof}
By \cite[Section §14]{HSTame} we conclude:
\begin{corollary}\label{FinalSchemes}
	Under the same assumptions, we have
	\[
	H^t_{Z,t}(X,\nu_m(n))=H^{t-r}_t(Z,\nu_m(n-r)).
	\] 
\end{corollary}
\begin{remark}
	\cref{MainResultCorollary} was proven in \cite{HTame} in the case where $n=0$ (Corollary 14.5 loc. cit.). 
	
	A proof when $\dim(X)=1$ is given in \cite[Theorem 1.2.]{KatharinaPurity}.	 
	
	Merici  applies \cref{MainResultCorollary} in his article \cite{Merici} in order to prove that tame cohomology is motivic.
\end{remark}

%% file: PurityRegular.tex
In this section, we prove cohomological purity for log de Rham-Witt sheaves in the tame topology for a regular pair over a base $S$ of characteristic $p$, again under the assumption of resolution of singularities. In our context, a pair of schemes $(X,Z)$ is called regular, if $Z$ is a closed subscheme of $X$, and $X$ and $Z$ are regular. 

 Note that the immersion $i:Z\to X$ is part of the structure, we choose however to omit it  in the notation. 
 
As usual, we say that $(X,Z)$ is of pure codimension $r$ if $Z$ is of pure codimension $r$ in $X$.

 This part  is inspired by the work of Shiho \cite{Shiho}, who generalized the result of Milne \cite{MilneValuesOfZeta} and Gros \cite{CohomologyLogGros} in the étale cohomology.
\begin{theorem}\label{regularCase}
	Let $i: Z\to X$ be a closed immersion of  schemes over a base scheme $S$ of characteristic $p$. Assume that $(X,Z)$ is a regular pair of pure codimension $r$. Then  in $D((Z/S)_t)$ we have the isomorphism 
	\[
	Ri^!\nu_{X,m}(n)\cong \nu_{Z,m}(n-r)[-r]
	\]
	under the assumption of resolution of singularities.
\end{theorem}
 \begin{remark}
	By applying \cref{rmk:tamevsstronglyet}, we also obtain the same purity result in the strongly \'etale topology.
 \end{remark}
 A key ingredient in the proof is the following result due to Popescu 
 \begin{proposition}[Popescu's Theorem {\cite[Theorem 07GC]{Stacks}}]\label{PopescuTheorem}
 	Let $B$ be a Noetherian ring and let $A$ be a regular $B$-algebra. Then $A$ is a filtered colimit of smooth $B$-algebras.
 \end{proposition}
 
An equivalent formulation is that for any diagram
 \begin{equation}
 	\begin{tikzcd}
 		B\ar[r]\ar[d]& A.\\
 		C\ar[ur]&\\
 	\end{tikzcd}
 \end{equation} 
with $C$ is a finitely generated $B$-algebra, there exists a smooth $B$-algebra $D$, fitting into a commutative diagram:
 \begin{equation}
 	\begin{tikzcd}
 		B\ar[r]\ar[d]& A\\
 		C\ar[ur]\ar[r]&D.\ar[u]\\
 	\end{tikzcd}
 \end{equation}
 
Assume now that the schemes $X$, $Z$ and $S$ are the affine schemes $\Spec(A)$, $\Spec(A/I)$ and $\Spec(R)$. Furthermore, since $Z\to X$ is regular, we may assume that $I$ is defined by a regular sequence $t_1,\dots,t_r$. 
 
 \begin{claim}\label{filteringDescriptionOfXZS}
 	Assume furthermore that $X$ is local. The chain of morphisms $Z\to X\to S$ can be expressed as a filtered limit of morphisms of affine schemes
 	\[
 	Z_i:=\Spec (A_i/I_i)\to X_i:=\Spec(A_i)\to S_i:=\Spec(R_i)
 	\]
 	where each $A_i$ and $R_i$ is of finite type over $\mF_p$, and $A_i$ and $A_i/I_i$ are smooth over $\mF_p$. 
 \end{claim}
 \begin{proof}
 	Observe first that the morphism $R\to A$ is the filtered colimit of morphisms $(R_i\to C_i)_{i\in \mathcal{I}}$, where $R_i$ and $C_i$ are of finite type over $\mF_p$. By \cref{PopescuTheorem}, find a smooth $\mF_p$ algebra $A_i$ such that the diagram
 	\[
 	\begin{tikzcd}
 		R\ar[r]&A&\\
 		R_i\ar[u]\ar[r]&C_i\ar[u]\ar[r]&A_i\ar[ul]
 	\end{tikzcd}
 	\] 
 	commutes. Thus $R\to A$ is the filtered colimit of $(R_i\to A_i)_{i\in \mathcal{I}}$. 
 	
 	In order to find adequate $I_i$'s, we first need a Lemma
 	\begin{lemma}[{\cite[Lemma C.7.]{DegliseRationalMotivic}}] \label{lem:FromRegularToSmoothGenerators}
 		Let $B\to A$ be a local homomorphism of regular local rings. Let $J\subsetneq B$ be a proper  ideal such that $A/JA$ is regular. Set $n=\dim A -\dim A/JA$. If there exists elements $f_1,\dots, f_n\in B$ generating $J$, then $f_i$ form a subset of regular parameters of $B$, and the ring $B/J$ is regular. 
 	\end{lemma}
 	\begin{proof}[Proof of the Lemma]
 		Let $\nn$ and $\mm$ be the respective maximal ideals of $B$ and $A$. Let $g_1,\dots,g_n$ be the images of $f_1,\dots,f_n$ along the morphism $B\to A$. Since $A/JA$ is regular,  $g_1,\dots,g_n$ can be extended to a regular sequence generating $\mm$ according to \cite[Lemma 00NR]{Stacks}.  This implies that $g_1,\dots,g_n$ are $A/\mm$-linearly independent elements of $\mm/\mm^2$, due to Nakayama's lemma \cite[Lemma 00DV]{Stacks}. Thus the $f_i$'S are $B/\nn$-linearly independent elements of $\nn/\nn^2$. By applying Nakayama's lemma again, we conclude the claim.
 	\end{proof}
 	 The ideal $I\subseteq A$ is defined by the regular sequence $t_1,\dots,t_r$, which extends to a regular sequence $(t_1,\dots,t_n)$, where $n=\dim X$.  By restricting the index set $\mathcal{I}$ and replacing each $A_i$ by a localization, we may lift, according to \cref{lem:FromRegularToSmoothGenerators}, the regular sequence $t_1,\dots, t_n$ to compatible regular sequences $t_{1,i},\dots,t_{n,i}$, for which each $A_i/(t_{1,i},\dots,t_{r,i})$ is smooth over $\mF_p$. Set $I_i:=(t_{1,i},\dots,t_{r,i})$, yielding the desired presentation.
 \end{proof} 
 
\begin{remark}\label{filteringRemark}
	\begin{itemize}
		\item  In \cref{filteringDescriptionOfXZS} one may drop the condition that $X$ is local. This has been done in \cite[Proposition C.8.]{DegliseRationalMotivic}. 
		\item \cref{filteringDescriptionOfXZS} implies that the closed immersion $i:\mZ:=\Spa(Z,S)\hookrightarrow \Spa(X,S)=:\mX$ may also be written as a limit of closed immersions $\mZ_i:=\Spa(Z_i,S_i)\hookrightarrow \Spa(X_i,S_i)=:\mX_i$, under the assumption that $X$ is local. Applying \cite[Corollary 5.4]{HTame} and excision we conclude that 
		\begin{equation}
			\colim_{i\in \mathcal{I}}H^s_{t,\mZ_i}(\mX_i,\nu_m(n))=H^s_{t,\mZ}(\mX,\nu_m(n)), \quad 
			\textrm{ for all } s\geq 0.
		\end{equation}
	\end{itemize}

	Since each $\mX_i=\Spa(X_i,S_i)$ and $\mZ_i=\Spa(Z_i,S_i)$ are open subspaces \break of  $\Spa(X_i,\mF_p)$ and $\Spa(Z_i,\mF_p)$, respectively, we may apply \cref{MainTheoremAtTheEnd} to each pair $(X_i,Z_i)$.
\end{remark}

We now remove the assumption that $X$, $Z$ and $S$ are affine. In order to prove \cref{regularCase}, we have to define a morphism 
\[
\Phi_{\mZ/\mX}: \nu_{\mZ,m}(n-r)[-r]\to Ri^! \nu_{\mX,m}(n)
\]
and verify that it is an isomorphism. 

\begin{proposition}
	Under the assumptions of \cref{regularCase}, $Ri^!\nu_{\mX,m}(n)$ is concentrated in degree $r$.
\end{proposition} 
\begin{proof}
	The claim is local on the tame site of $\mX$. We can thus assume that $\mX=\Spa(A,A^+)$ is a tamely henselian adic space. By assumption, a regular ideal $I\subseteq A$ exists, such that $\mZ=\Spa(A/I,A^+)=\Spa(A/I,A^+/I)$. According to \cref{filteringRemark}, we may describe $\Spa(A,A^+)$ and $\Spa(A/I,A^+)$ as filtered limits of adic spaces $\Spa(A_i,R_i)$ and $\Spa(A_i/I_i,R_i)$, where $A_i$ and $R_i$ are of finite type over $\mF_p$, and $A_i$  and $A_i/I_i$ are smooth $\mF_p$-algebras. 
	
	For $s\neq r$, the sheaf $R^si^!\nu_{m}(n)$ corresponds to the cohomology group $H^s_{t,\mZ}(\mX,\nu_{m}(n))$. By \cref{filteringRemark}, the latter is isomorphic to 
	\begin{equation}\label{CohomologyColimit}
			\colim_i H^s_{t,\Spa(A_i/I_i,R_i)}(\Spa(A_i,R_i),\nu_m(n))=\colim_i H^{s-r}_{t}(\Spa(A_i/I_i,R_i),\nu_m(n-r))
	\end{equation}
	By \cite[Corollary 5.4]{HTame}, \eqref{CohomologyColimit} is isomorphic to $H^{s-r}_t(\Spa(A/I,A^+),\nu_m(n-r))$, which is trivial, since $\Spa(A/I,A^+)=\Spa(A/I,A^+/I)$ is tamely henselian.
\end{proof}

\begin{proof}[Proof of \cref{regularCase}]

It is only left to construct an isomorphism  of sheaves in $D((Z/S)_t)$.
\[
\Phi_{\mZ/\mX}: \nu_{\mZ,m}(n-r)[-r]\to R^ri^! \nu_{\mX,m}(n).
\]
We use the following construction from  \cite[Section 3]{Shiho} on $W_m\Omega_X^n$.
\begin{DefLem}\label{DefLemAboutWOmega}
	For each $n>0$ There exists a morphism 
\[
		\Phi_{Z/X}: W_m\Omega_Z^{n-r}\to R^ri^!W_m\Omega_X^n.
\]
	Locally, when $Z$ and $X$ are the colimit of affine schemes $Z_i$ and $X_i$ which are smooth over $\mF_p$, the morphism $\Phi_{Z/X}$ is the colimit of the morphisms
	\[
	\Phi_{Z_i/X_i}: W_m\Omega_{Z_i}^{n-r}\to R^ri^!W_m\Omega_{X_i}^{n}
	\]
	defined by Gros in \cite[Subsection 3.4]{CohomologyLogGros}.
\end{DefLem}

By \cite[Lemma 8.3]{HTame}, we obtain a morphism
\begin{equation}
	\Phi_{\mZ/\mX}: W_m\Omega_{\mZ}^{n-r}\to R^ri^!W_m\Omega_{\mX}^n.
\end{equation}

We have established at the end of the proof of \cref{MainTheoremAtTheEnd} that the natural morphism $R^ri^!\nu_{\mX_i,m}(n)\to R^ri^!W_m\Omega_{\mX_i}^n$ is injective.  By \cref{DefLemAboutWOmega}, we conclude that $R^ri^!\nu_{\mX,m}(n)\to R^ri^!W_m\Omega_{\mX}^n$ is also injective: this is checked on local Huber pairs, where the morphism is a filtered colimit of injective maps.

It suffices to check that the composition of morphisms
\[
\nu_m(n-r)\hookrightarrow W_m\Omega_Z^{n-r}\xrightarrow{\Phi_{\mZ/\mX}} R^ri^!W_m\Omega_Z^{n}.
\]
factors through $R^ri^!\nu_m(n)$ and that $\Phi_{\mZ/\mX}:\nu_m(n-r)\to R^ri^!\nu_m(n)$ is an isomorphism.  This claim may be checked on local adic spaces. By expressing $\mX$ and $\mZ$ as colimits of smooth adic spaces as in \cref{filteringRemark}, the claim follows by \cref{MainTheoremAtTheEnd}.
\end{proof}
\begin{remark}
	In \cite[Theorem 3.2]{Shiho}, Shiho required $X$ to be part of \emph{the category $\mathcal{C}$} whose objects are excellent schemes $X$ satisfying $[\kappa(x):\kappa(x)^p]<\infty$ for all $x\in X$. These conditions are not required in our context, but are necessary in the étale cohomology case  \cite[Remark 3.10]{Shiho}.
\end{remark}

%% file: Resolution.tex
Let $\mathcal{X} := \Spa(X, S)$ be a discretely ringed adic space, where $X$ and $S$ are qcqs schemes, and $X \to S$ is separated. In our applications, we often seek an étale cover of $\mathcal{X}$ by affinoids of the form $\Spa(U, \bar{U})$, where $U \subset \bar{U}$ is an open immersion. Such opens form a basis for the topology on $\mathcal{X}$. This construction relies on Nagata's compactification theorem and its refinement by Conrad \cite{Conrad}. This is the subject of \cref{sec:Compactifications}.

Our goal is to replace pairs $(U, \bar{U})$ by log regular log schemes, since logarithmic differentials admit a simpler and more explicit description in this context (see \cite{LogDiff}). However, this approach runs into the classical problem of resolution of singularities, which remains open in dimension $>3$ in positive characteristic.

We recall the following version of Nagata's compactification theorem:

\begin{proposition}[{\cite[Theorem 4.3.]{Conrad}}]\label{NagataCompactification}
	Let $S$ be a qcqs scheme, and let $X$ be a separated $S$-scheme of finite type. Then there exists an open immersion $j: X \hookrightarrow \bar{X}$ over $S$ such that $\bar{X} \to S$ is proper. We call such an immersion a \emph{compactification} of $X$ over $S$.
\end{proposition}

This implies that the topology on a discretely ringed adic space $\mathcal{X} = \Spa(X, S)$, where $X \to S$ is separated and of finite type, is generated by affinoids $\Spa(U, \bar{U})$, where $U \subset \bar{U}$ is an open immersion. After replacing $\bar{U}$ by a blow-up along $\bar{U} \setminus U$, we may assume that the morphism $U \to \bar{U}$ is affine. Indeed, the complement $Z := \bar{U} \setminus U$ is a Cartier divisor, and one can cover $\bar{U}$ by affine opens $\Spec(A)$ where $Z \cap \Spec(A) = V(f)$ for some $f \in A$, so $U \cap \Spec(A) = \Spec(A[f^{-1}])$.

Given a separated morphism $X \to S$ of qcqs schemes, the category of compactifications forms a filtered system. In particular, if $X \hookrightarrow \bar{X}_1$ and $X \hookrightarrow \bar{X}_2$ are two compactifications, then the induced map $X \to \bar{X}_1 \times_S \bar{X}_2$ is also a compactification.

\begin{proposition}[{\cite[Theorem 2.11]{Conrad}}]\label{BlowupsDomModif}
	Let $f: X' \to X$ be a proper morphism of qcqs schemes that restricts to an isomorphism on a quasi-compact dense open subscheme $U \subseteq X, X'$. Then there exist $U$-admissible blow-ups $\bar{X} \to X$ and $\bar{X}' \to X'$ and an isomorphism $\bar{X}' \cong \bar{X}$ over $f$.
\end{proposition}
In particular, the category of compactifications is dominated by those in which all morphisms are blow-ups.
\section{Resolutions}

Let $X \to S$ be a smooth separated morphism of finite type, and assume $S = \Spec(k)$ for a perfect field $k$ of characteristic $p$. Then $X$ is quasi-excellent. We would like to find a smooth compactification $\bar{X}$ of $X$ such that $\bar{X} \setminus X$ is an snc divisor. Ideally, we also want every such compactification to be dominated by a sequence of blow-ups in smooth centers lying within the boundary. This allows us to apply  \cref{AcyclicityProposition} to compute cohomology of logarithmic differentials. These goals remain conjectural in dimensions $>3$.

\begin{conjecture}\label{ConjResolution3}
	Let $\bar{X}$ be a reduced, separated, Noetherian, quasi-excellent scheme. Then there exists a proper birational morphism $\pi: \bar{X}' \to \bar{X}$ such that:
	\begin{itemize}
		\item $\bar{X}'$ is regular;
		\item $\pi$ is an isomorphism over the regular locus of $\bar{X}$;
		\item The exceptional divisor $\pi^{-1}(\bar{X}_{\mathrm{sing}})$ is an snc divisor in $\bar{X}'$.
	\end{itemize}
\end{conjecture}

\cref{ConjResolution3} holds when $\bar{X}$ is of dimension $\leq 3$ \cite{Resolution3}. Now back to our situation, where $X\to S$ is separated and smooth and assume that $S=\Spec(k)$, where $k$ is a perfect field of characteristic $p$. Then $X$ is quasi-excellent. Under the assumption of \cref{ConjResolution3}, one can find a smooth compactification $\bar{X}$ of $X$ such that $\bar{X} \setminus X$ is an snc divisor, and hence $(X, \bar{X})$ is a log smooth pair.

\begin{conjecture}[Embedded resolution]\label{ConjResolution2}
	Let $Z$ be a reduced, excellent, Noetherian scheme, and $X$ a regular excellent scheme. Suppose $Z \hookrightarrow X$ is a closed immersion. Then there exists a projective, surjective morphism $\pi: X' \to X$, which is an isomorphism over $X \setminus Z$, such that $\pi^{-1}(Z)$ is an snc divisor in $X'$.
\end{conjecture}

 This is known when $\dim(Z) \leq 2$ \cite[Corollary 1.5]{Resolution2}. Together, Conjectures \ref{ConjResolution3} and \ref{ConjResolution2} imply that every smooth discretely ringed adic space $\Spa(X, S)$ over a perfect field $k$ of characteristic $p$ has a cover by affinoids $\Spa(U, \bar{U})$ such that $\bar{U}$ is smooth and $\bar{U} \setminus U$ is an snc divisor.
 This holds up to dimension $3$ since \cref{ConjResolution2} holds up to dimension $2$ and \cref{ConjResolution3} holds up to dimension $3$.

We present a stronger version of the conjecture, which implies the Conjecture of embedded resolution when $X$ is a defined over a perfect field $k$.
\begin{conjecture}[Embedded resolution of singularities with boundary]\label{GeneralResolution}
	Consider a smooth scheme $\bar{X}$ over a perfect field $k$. Let $D\subset \bar{X}$ be an snc divisor. Let $\bar{Z}\subset \bar{X}$ be a reduced closed subscheme, which is smooth over $\bar{X}\backslash D$, s.t. $\bar{Z}\backslash D$ is smooth over $k$. Then the triple $(\bar{X},D,\bar{Z})$ admits a resolution $\pi:\bar{Y}\to \bar{X}$ by blowups in smooth centers
	\[
	\bar{Y}:=\bar{X}_m\to \bar{X}_{m-1}\to \dots \to \bar{X}_2 \to \bar{X}_1:=\bar{X},
	\]
	satisfying the following conditions
	\begin{itemize}
		\item Each morphism $\bar{X}_{i+1}\xrightarrow{\pi_i} \bar{X}_{i}$ is a blowup in a smooth center $W_i$ with exceptional divisor $E_{i+1}$. Set $D_1:=D$, and let $D_{i+1}$ be the strict transform of $D_i$ in $\bar{X}_{i+1}$.
		\item $D_{i}$ is an snc divisor and $W_i \subset D_i$.
		\item $\pi^{-1}(\bar{Z})$ is a smooth $k$-scheme.
	\end{itemize}
\end{conjecture}
\cref{GeneralResolution} is proven when $\bar{Z}$ has dimension $\leq 2$ \cite[Theorem 6.9]{Resolution2}. 

\begin{remark}\label{DescriptionOfOpens}
	Assuming Conjectures \ref{ConjResolution3} and \ref{ConjResolution2}, we conclude that opens of the form $\Spa(U, \bar{U})$, where $(U, \bar{U})$ is a log smooth pair, form a basis for the topology on $\Spa(X, S)$. By further covering $\bar{U}$ by affine opens, we may assume that these are affinoids. Assuming \cref{GeneralResolution}, any $U$-modification of $\bar{U}$ is dominated by a sequence of blow-ups in smooth centers.
	
	Indeed, modifications are dominated by blow-ups \cite[Corollary 3.4.8]{TemkinRZ}. Moreover, in our setting, $U$-blow-ups and $U$-admissible blow-ups coincide \cite[Remark 3.4.5(i)]{TemkinRZ}. Applying \cref{GeneralResolution} to the center $Z \subset \bar{U} \setminus U$ allows us to resolve $Z$ within the boundary divisor, ultimately yielding a smooth blow-up dominating any given $U$-modification.
\end{remark}

%% file: litteratur.bib
@article{KSTVoerst, title={Towards Vorst’s conjecture in positive characteristic}, volume={157}, DOI={10.1112/S0010437X21007120}, number={6}, journal={Compositio Mathematica}, author={Kerz, Moritz and Strunk, Florian and Tamme, Georg}, year={2021}, pages={1143–1171}}

@book{Ogus, place={Cambridge}, series={Cambridge Studies in Advanced Mathematics}, title={Lectures on Logarithmic Algebraic Geometry}, DOI={10.1017/9781316941614}, publisher={Cambridge University Press}, author={Ogus, Arthur}, year={2018}, collection={Cambridge Studies in Advanced Mathematics}}

@misc{Stacks,
	shorthand    = {Stacks},
	author       = {The {Stacks Project Authors}},
	title        = {\textit{Stacks Project}},
	howpublished = {https://stacks.math.columbia.edu},
	year         = {2018},
}

@book{Esnault,
	title={Lectures on vanishing theorems},
	author={Esnault, H{\'e}l{\`e}ne and Viehweg, Eckart},
	volume={20},
	year={1992},
	publisher={Springer}
}

@article{Conrad,
	title={Deligne's notes on Nagata compactifications},
	author={Conrad, Brian},
	journal={Journal-Ramanujan Mathematical Society},
	volume={22},
	number={3},
	pages={205},
	year={2007},
	publisher={THE RAMANUJAN MATHEMATICAL SOCIETY}
}

@misc{LogDiff,
  title={Logarithmic differentials on discretely ringed adic spaces}, 
author={Katharina Hübner},
year={2024},
eprint={2009.14128},
archivePrefix={arXiv},
primaryClass={math.AG},
url={https://arxiv.org/abs/2009.14128}, 
}

@InProceedings{TemkinMetrization,
	author={Michael Temkin},
	editor={Baker, Matthew
	and Payne, Sam},
	title={Metrization of Differential Pluriforms on Berkovich Analytic Spaces},
	booktitle={Nonarchimedean and Tropical Geometry},
	year={2016},
	publisher={Springer International Publishing},
	address={Cham},
	pages={195--285},
	isbn={978-3-319-30945-3}
}

@article{MilneDuality,
	author = {Milne, J. S.},
	journal = {Annales scientifiques de l'École Normale Supérieure},
	language = {eng},
	number = {2},
	pages = {171-201},
	publisher = {Elsevier},
	title = {Duality in the flat cohomology of a surface},
	url = {http://eudml.org/doc/81978},
	volume = {9},
	year = {1976},
}

@article{sato2007logarithmic,
	title={Logarithmic Hodge--Witt sheaves on normal crossing varieties},
	author={Sato, Kanetomo},
	journal={Mathematische Zeitschrift},
	volume={257},
	pages={707--743},
	year={2007},
	publisher={Springer}
}

@article{HSTame,
	title={The tame site of a scheme},
	author={Hübner, Katharina and Schmidt, Alexander},
	journal={Inventiones mathematicae},
	volume={223},
	pages={379--443},
	year={2021},
	publisher={Springer}
}

@article{HTame,
	title={The adic tame site},
	author={Hübner, Katharina},
	journal={Documenta Mathematica},
	volume={26},
	pages={873--945},
	year={2021}
}

@article{TemkinRZ,
	title={Relative Riemann-Zariski spaces},
	volume={185},
	ISSN={1565-8511},
	urladdr={http://dx.doi.org/10.1007/s11856-011-0099-0},
	DOI={10.1007/s11856-011-0099-0},
	number={1},
	journal={Israel Journal of Mathematics},
	publisher={Springer Science and Business Media LLC},
	author={Temkin, Michael},
	year={2011},
	month=sep, pages={1–42} }

@article{Resolution3,
	title = {Resolution of singularities of arithmetical threefolds},
	journal = {Journal of Algebra},
	volume = {529},
	pages = {268-535},
	year = {2019},
	issn = {0021-8693},
	doi = {https://doi.org/10.1016/j.jalgebra.2019.02.017},
	url = {https://www.sciencedirect.com/science/article/pii/S0021869319301061},
	author = {Vincent Cossart and Olivier Piltant}
}

@book{Resolution2,
title={Desingularization: Invariants and strategy},
author={Cossart, Vincent and Jannsen, Uwe and Saito, Shuji},
volume={2270},
year={2020},
publisher={Springer}
}

@book{HuberEtale,
	title={{\'E}tale cohomology of rigid analytic varieties and adic spaces},
	author={Huber, Roland},
	volume={30},
	year={2013},
	publisher={Springer}
}

@article{MilneValuesOfZeta,
	ISSN = {00029327, 10806377},
	URL = {http://www.jstor.org/stable/2374676},
	author = {J. S. Milne},
	journal = {American Journal of Mathematics},
	number = {2},
	pages = {297--360},
	publisher = {Johns Hopkins University Press},
	title = {Values of Zeta Functions of Varieties Over Finite Fields},
	urldate = {2024-07-10},
	volume = {108},
	year = {1986}
}

@InCollection{FujiwaraPurity,
	Author = {Fujiwara, Kazuhiro},
	Title = {A proof of the absolute purity conjecture (after {Gabber}).},
	BookTitle = {Algebraic geometry 2000, Azumino. Proceedings of the symposium, Nagano, Japan, July 20--30, 2000},
	ISBN = {4-931469-20-5},
	Pages = {153--183},
	Year = {2002},
	Publisher = {Tokyo: Mathematical Society of Japan},
	Language = {English},
	zbMATH = {1944554},
	Zbl = {1059.14026}
}

@article{KatharinaPurity,
	title={Tame and strongly {\'e}tale cohomology of curves},
	author={Katharina Hübner},
	journal={Israel Journal of Mathematics},
	volume={253},
	number={1},
	pages={1--42},
	year={2023},
	publisher={Springer}
}

@article{DegliseRationalMotivic,
	shorthand = {DFJK21},
	title={On the rational motivic homotopy category},
	volume={8},
	ISSN={2270-518X},
	url={http://dx.doi.org/10.5802/jep.153},
	DOI={10.5802/jep.153},
	journal={Journal de l’École polytechnique — Mathématiques},
	publisher={Cellule MathDoc/CEDRAM},
	author={Déglise, Frédéric and Fasel, Jean and Jin, Fangzhou and Khan, Adeel A.},
	year={2021},
	month=feb, pages={533–583} }

@misc{HTCurves,
	title={Adic curves: stable reduction, skeletons and metric structure}, 
	author={Katharina Hübner and Michael Temkin},
	year={2024},
	eprint={2406.07414},
	archivePrefix={arXiv},
	primaryClass={math.AG},
	urladdr={https://arxiv.org/abs/2406.07414}, 
}

@misc{Merici,
	title={Tame motivic cohomology and a conjecture of H\"ubner--Schmidt}, 
	author={Alberto Merici},
	year={2024},
	eprint={2408.02499},
	archivePrefix={arXiv},
	primaryClass={math.AG},
	url={https://arxiv.org/abs/2408.02499}, 
}

@misc{KelleyCohomology,
	title={Hodge cohomology with a ramification filtration, II}, 
	author={Shane Kelly and Hiroyasu Miyazaki},
	year={2023},
	eprint={2306.06864},
	archivePrefix={arXiv},
	primaryClass={math.AG},
	urladdr={https://arxiv.org/abs/2306.06864}
}

@book{HartshorneAmple,
	title={Ample subvarieties of algebraic varieties},
	author={Hartshorne, Robin},
	volume={156},
	year={2006},
	publisher={Springer}
}

@book{CohomologyLogGros,
	author = {Gros, Michel},
	title = {Classes de {Chern} et classes de cycles en cohomologie de {Hodge-Witt} logarithmique},
	series = {M\'emoires de la Soci\'et\'e Math\'ematique de France},
	publisher = {Soci\'et\'e math\'ematique de France},
	number = {21},
	year = {1985},
	doi = {10.24033/msmf.322},
	mrnumber = {87m:14021},
	zbl = {0615.14011},
	language = {fr},
	url = {http://www.numdam.org/item/MSMF_1985_2_21__1_0/}
}

@misc{morrowLogHodgeWitt,
	title={$K$-theory and logarithmic Hodge-Witt sheaves of formal schemes in characteristic $p$}, 
	author={Matthew Morrow},
	year={2015},
	eprint={1512.04703},
	archivePrefix={arXiv},
	primaryClass={math.KT},
	url={https://arxiv.org/abs/1512.04703}, 
}

@book{MilneEtale,
	title={Etale cohomology (PMS-33)},
	author={Milne, James S},
	year={1980},
	publisher={Princeton university press}
}

@misc{WedhornAdicSpaces,
	title={Adic Spaces}, 
	author={Torsten Wedhorn},
	year={2019},
	eprint={1910.05934},
	archivePrefix={arXiv},
	primaryClass={math.AG},
	url={https://arxiv.org/abs/1910.05934}, 
}

@article{Shiho,
	title={On Logarithmic Hodge-Witt Cohomology of Regular Schemes},
	author={Shiho, Atsushi},
	journal={Journal of Mathematical Sciences-University of Tokyo},
	volume={14},
	number={4},
	pages={567},
	year={2007},
	publisher={Tokyo: The University, c1994-}
}

@article{KatzNilpotent,
	title={Nilpotent connections and the monodromy theorem: Applications of a result of Turrittin},
	author={Katz, Nicholas M},
	journal={Publications math{\'e}matiques de l'IHES},
	volume={39},
	pages={175--232},
	year={1970}
}

@misc{HubnerBaseChange,
title={Tame proper base change for discretely ringed adic spaces}, 
author={Katharina Hübner},
year={2025},
eprint={2503.13312},
archivePrefix={arXiv},
primaryClass={math.AG},
url={https://arxiv.org/abs/2503.13312}
}

@misc{FujiwaraFoundationsRigidGeometry,
      title={Foundations of Rigid Geometry I}, 
      author={Kazuhiro Fujiwara and Fumiharu Kato},
      year={2017},
      eprint={1308.4734},
      archivePrefix={arXiv},
      primaryClass={math.AG},
      url={https://arxiv.org/abs/1308.4734}, 
}
